\documentclass[11pt, reqno]{amsart}
\usepackage{lipsum}
\usepackage{a4wide}
\usepackage{amssymb} 
\usepackage{amsmath}
\usepackage{amsthm} 
\usepackage[all,cmtip]{xy}
\usepackage{tikz} 
\usepackage{amscd}
\usepackage{marginnote}
\usepackage{hyperref}
\usepackage{enumitem}
\hypersetup{colorlinks,linkcolor={blue},citecolor={blue},urlcolor={red}}
\theoremstyle{plain}
\numberwithin{equation}{section}

\newcommand{\gl}{\operatorname{GL}}

\newcommand{\tr}{\operatorname{tr}}
\newcommand{\End}{\operatorname{End}}

\newcommand{\g}{\mathfrak{g}}
\newcommand\Endo{\textup{End}}

\newcommand{\Z}{\mathbb{Z}}
\newcommand\N{\mathbb{N}}
\newcommand\tensor{\otimes}
\newcommand{\dsum}{\oplus}
\newcommand\ie{{\it i.e.,\,}}
\newcommand\C{\mathbb{C}}
\newcommand\Sym{\mathfrak{S}}

\newtheorem{theorem}{Theorem}[section]
\newtheorem{corollary}[theorem]{Corollary} 
\newtheorem{lemma}[theorem]{Lemma}
\newtheorem{remark}[theorem]{Remark}
\newtheorem{proposition}[theorem]{Proposition}

\newcommand\blfootnote[1]{%
  \begingroup
  \renewcommand\thefootnote{}\footnote{#1}%
  \addtocounter{footnote}{-1}%
  \endgroup
}

\begin{document}

\title[Graded Invariants for mixed tensor spaces]{Graded Picture Invariants and polynomial invariants for mixed tensor superspaces}
\author{Santosha Pattanayak, Preena Samuel}
\date{\today}

\maketitle

\begin{abstract}
 In this paper we consider the mixed tensor space of a $\mathbb Z_2$-graded vector space. We obtain a spanning set of invariants of the associated symmetric algebra under the action of the general linear supergroup as well as the queer supergroup over the Grassmann algebra. As a consequence, we give a generating set of polynomial invariants for the simultaneous adjoint action of the general linear supergroup  on several copies of its Lie superalgebra. We show that in this special case, these turn out to be  the supertrace monomials which is
analogous to the result in \cite{P}. A queer supergroup analogue of these results is also obtained.   
\end{abstract}

\section{Introduction}
    The theory of Lie superalgebras has aroused much interest both in mathematics and physics. Lie superalgebras and their representations play a fundamental role in theoretical physics since they are used to describe supersymmetry in a mathematical framework. A comprehensive description of the mathematical theory of Lie superalgebras is given in \cite{Kac}, containing the complete classification of all finite-dimensional simple Lie superalgebras over an algebraically closed field of characteristic zero. Since then, these superalgebras have found applications in various areas including quantum mechanics, nuclear physics, particle physics, and string theory. In the last few years, the theory of Lie superalgebras has experienced a remarkable evolution with many results obtained on their representation theory and classification, most of them extending well-known facts from the theory of Lie algebras.

    In \cite{P}, Procesi studied tuples $(A_1, \cdots ,A_k)$ of endomorphisms of a finite-dimensional vector space up to simultaneous conjugation by studying the corresponding ring of invariants. He showed that for an algebraically closed field $F$ of characteristic zero the algebra of invariants $F[(A_i)_{jl}]^{GL_r(F)}$ can be generated by traces of monomials in $A_1, \cdots A_k$. 
The main tool used in the work of Procesi, in order to describe the invariants, is the Schur–Weyl duality. This tool was used in a similar way also to study more complicated algebraic structures than just a vector space equipped with endomorphisms. In \cite{dks}, Datt, Kodiyalam and Sunder applied this machinery to the study of finite-dimensional complex semisimple Hopf algebras. They were able to obtain a complete set of invariants for separating isomorphism classes of complex semi-simple Hopf algebras of a given dimension. This  they  accomplished by giving an explicit spanning set for the invariant ring of the mixed tensor space. These are called  ``picture invariants". These invariants are also obtained by using techniques from Geometric Invariant Theory in \cite{M}. In \cite{KR} these picture invariants were used to describe a finite collection of rational functions in the structure constants of a Lie algebra, which form a complete set of invariants for the isomorphism classes of complex semisimple Lie algebras of a given dimension. More recently, the invariant ring of the mixed tensor spaces was used to extend results of \cite{dks} to separate isomorphism classes of complex Hopf algebras of small dimensions \cite{PS}. The invariant ring of the mixed tensor space arises naturally while obtaining invariants of any such isomorphism classes. Here, we look at the analogue of this invariant ring in the supersetting. We define ``graded picture invariants" for the invariant ring of the mixed tensor superspace in the supersetting and show that these invariants $\Lambda$-linearly span the invariant ring. 

  The invariant theory of superalgebras has seen remarkable progress in the recent years. In \cite{Ser}, Sergeev proved fundamental theorems of invariant
theory for classical supergroups and their Lie superalgebras (see \cite{Ser1} and \cite{Ser2} for a detailed exposition). A super analogue of the Schur–Weyl duality between
the general linear superalgebra and the symmetric group was established independently
in \cite{BR} and \cite{Ser}. In \cite{SM}, Shader and Moon gave a super Schur–Weyl duality for the general linear superalgebra acting on the mixed tensor space.

The general linear supergroup of a super vector space $V$ over $\C$ is defined as a super group scheme. We, however, take the popular view of considering this group scheme as an actual group over the Grassmann algebra $\Lambda$ by taking the $\Lambda$-valued points of the general linear group scheme. We denote this group as $\gl(U)$ where $U:=V\tensor_\C\Lambda$. The queer supergroup $Q(U)$ is defined as a subgroup of $\gl(U)$. By working with $\gl(U)$ and $Q(U)$ we are able to take a more elementary approach to the theory of invariants for these supergroups.  

In \cite{Berele}, \cite{Berele2}, Berele succesfully extended results of \cite{P} to the supersetting by showing that the invariants for the $\gl(U)$-action on the super ring of polynomials in the entries of $d$-tuples of matrices over $\Lambda$ is generated by supertrace monomials in the $d$ matrices.   As in the classical case, the main ingredient for obtaining the invariants comes from the super analogue of the Schur-Weyl duality. In this paper, we extend Berele's results to the action of the general linear supergroup and the queer supergroup on the mixed tensor superspace $\dsum_{i=1}^sU_{b_i}^{t_i}$ where  $t_i$, $b_i$ are in $\N\cup \{0\}$ for all $i=1,\ldots s$. The case of invariants of $d$ copies of matrices, as in \cite{Berele2} may be seen as a special case of the above by taking $t_i=1=b_i$ for all $i=1\ldots,s$. We show in Theorem~\ref{t:picture} that the ring of invariants of the supersymmetric algebra $S(\dsum_{i=1}^sU_{b_i}^{t_i})^*$, is generated by certain special invariants which are analogues of the ``picture invariants" in \cite{dks}. We then show in Theorem~\ref{t:trace} that in the special case of $t_i=1=b_i$ for all $i=1\ldots,s$, this agrees with the results of \cite{Berele2}.

  In the supersetting, we use the notion of $\Lambda$-valued polynomials over $U$, as introduced in \cite{LZ}. This is done via the supersymmetric algebra of the dual space $U^*.$ We consider the $\Lambda$-module of maps from the degree-$0$ component of $U$ to $\Lambda$, denoted as  $\mathcal{F}(U_0,\Lambda)$. Then using an analogue of the restitution map from the $r$-fold tensor space of $U^*$ to  $\mathcal{F}(U_0, \Lambda)$, as defined in \cite{LZ}, we call the image of this restitution map to be the space of homogeneous polynomials of degree $r$. The polynomial ring on $U$ is then taken to be the direct sum of these spaces. It is also shown in {\it loc.\,cit} that this polynomial ring is isomorphic to the symmetric algebra of $U^*$, under the restitution map. We view this description of polynomials over $U$ as a suitable alternative to Berele's notion of a polynomial over matrices with scalars in $\Lambda$, as defined in \cite{Berele2}, since we are interested in mixed tensor superspaces which do not have any natural identification with matrices. In this paper, we use the above notion of $\Lambda$-valued polynomials on $\dsum_{i=1}^sU_{b_i}^{t_i}$ and obtain a generating set for the polynomial ring of invariants for a mixed tensor superspace. We obtain `graded picture invariants' as invariant polynomial functions on the mixed tensor space. We then show that in the special case when the mixed tensor space corresponds to several copies of the endomorphism space of $U$, the graded picture invariants are just supertrace monomials as given in \cite{Berele2}. This may be regarded as the superanalogue of Procesi's result.

  The queer Lie superalgebras form a family of ``strange" Lie superalgebras
since they do not possess non-degenerate invariant even bilinear forms. The representation theory and invariant theory of the queer superalgebras $\mathfrak q(n)$ have been intensively studied in the last two decades. The analogous statement of Schur-Weyl duality for the queer superalgebra was given by Sergeev \cite{Ser3} during the study of their polynomial representations. The duality is often called the Schur–Weyl–Sergeev duality which gives an isomorphism between a super-extension of the symmetric group $Ser_k$, called the Sergeev algebra and $End_{\mathfrak q(n)}(V^{\otimes k})$ for $n \geq k$. This duality also lifts to the queer supergroup $Q(U)$ and the Sergeev superalgebra (see \cite{Berele2}). An extension of this duality to mixed tensor spaces can be found in \cite{JK}. Using the Schur-Weyl-Sergeev duality, we show in Theorem~\ref{t:queer} that the polynomial ring of invariants of the mixed tensor space $\dsum_{i=1}^sU_{b_i}^{t_i}$ under the action of the queer supergroup is generated by a larger collection of ``graded picture invariants". As a corollary, we show that the polynomial invariants in the case of $t_i=1=b_i$ for all $i=1,\ldots ,s$ is generated by queer trace monomials, which is also proved in \cite{Berele2}. 

We now give an outline of the paper. In section 2 we review preliminaries of super vector spaces and Lie superalgebras and we also recall the super analogue of the classical Schur-Weyl duality. In section 3 we include a proof of the first fundamental theorem of invariant theory for tensors. In section 4 we introduce the notion of graded picture invariants and prove that these graded picture invariants span the invariants of the supersymmetric algebra. Using this result we give a spanning set for the polynomial invariants of the mixed tensor space and thereby show that the supertrace monomials span the  polynomial invariants for the simultaneous adjoint action of the general linear supergroup on several copies of its Lie superalgebra. In section 5 we recall the definition of Queer supergroup and state the Schur-Weyl-Sergeev duality for it. In section 6 we extend the results of section 4 to the Queer supergroup.

\section{Preliminaries}
\subsection{Grassmann algebra}
Given a vector space $M$ we denote by $T(M)$ the tensor algebra on $M$. Let $I$ be the $2$-sided ideal of $T(M)$ generated by the elements $v \otimes w+w \otimes v$, $v,w \in M$. Then the exterior algebra over $M$ is the $\mathbb Z_+$-graded algebra defined by $\Lambda(M)=T(M)/I$. As a $\mathbb C$-vector space, $\Lambda(M)$ has dimension $2^n$, where $n=dim(M)$. Note that $\Lambda(M)$ is $\mathbb Z_2$-graded in which $\Lambda_0$ (resp $\Lambda_1$) is the direct sum of the homogeneous subspaces of even (resp. odd) degrees. This superalgebra is called the Grassmann algebra of degree $n$. Choosing a basis $v_1,v_2, \cdots, v_n$ of $M$, the set $\{v_1^{\alpha_1}v_2^{\alpha_2},\cdots ,v_n^{\alpha_n}: \alpha_i=0,1\}$ forms a basis of $\Lambda(M)$. Since $\Lambda(M)$ depends only on $n=dim(M)$ up to isomorphism, we write $\Lambda(n)$ for $\Lambda(M)$. 

For any flag $0 \subset M_1 \subset M_2 \subset \cdots \subset M_{n-1} \subset M_n=M$, where $dim(M_i)=i$, we have $\mathbb C \subset \Lambda(1) \subset \Lambda(2) \subset \cdots \subset \Lambda(n-1) \subset \Lambda(n)$. Then the infinite dimensional Grassmann algebra is defined by the direct limit of this direct system, $\Lambda=\varinjlim \Lambda(n)$. Apart from the $\mathbb Z_2$-grading, $\Lambda$ has a $\mathbb Z_+$-grading $\Lambda=\oplus_{i=0}^\infty\Lambda_i$, where $\Lambda_0=\mathbb C$. The Grassmann algebra $\Lambda$ is super commutative means that if $x,y \in \Lambda$ are homogeneous, then $xy=(-1)^{|x||y|}yx$, where $|x|=i$ if $x \in \Lambda_i$. This implies that any right or left module over $\Lambda$ is automatically a bimodule.

Throughout this paper the modules that we consider are over the infinite dimensional Grassmann algebra $\Lambda$. The projection map $\Lambda \rightarrow \mathbb C$ will be used to specialize the scalars to $\mathbb C$.

\subsection{Super vector spaces} A $\mathbb Z_2$-graded vector space $V$ over $\mathbb C$ is a direct sum $V_0 \oplus V_1$ of two subspaces, the even subspace $V_0$ and the odd subspace $V_1$. If $\mathrm{dim}(V_0)=k$ and $\mathrm{dim}(V_1)=l$, we write $\mathrm{sdim}(V)=(k|l)$, and call it the {\bf superdimension} of $V$. Define the parity $|v|$ of any homogeneous element $v \in V_i$ by $|v|=i$. 

For any $\mathbb Z_2$-graded vector space $V$ over $\mathbb C$, we let $U:=V \otimes \Lambda$ be the corresponding free $\Lambda$-module. We have the natural embedding $V \hookrightarrow U, v \mapsto v \otimes 1$. If $W$ is another $\mathbb Z_2$-graded $\mathbb C$-vector space and $U':=W \otimes \Lambda$, then ${\rm Hom}_{\Lambda}(U, U') \cong {\rm Hom}_{\mathbb C}(V,W) \otimes \Lambda$. 

Let $M_{k,l}(\Lambda)$ be the set of all $(k+l) \times (k+l)$ matrices over $\Lambda$. We have $M_{k,l}(\Lambda) \cong M_{k,l}(\mathbb C) \otimes \Lambda$ as $\mathbb Z_2$-graded vector spaces. We write a matrix $T$ in block form as  \begin{equation}
    T=\begin{pmatrix}
A & B \\
C & D
\end{pmatrix}
\end{equation} If  $T$ is even (resp. odd), then the entries of $A$ and $D$ belong to $\Lambda_0$ (resp. $\Lambda_1$), and the entries of $B$ and $C$ belong to $\Lambda_1$ (resp. $\Lambda_0$) . If $sdim(V)=(k|l)$ and $U:=V \otimes \Lambda$, we have a $\mathbb Z_2$-graded vector space isomorphism $\Endo_{\Lambda}(U) \cong M_{k,l}(\Lambda).$ 

We define the supertrace function $str: M_{k,l}(\Lambda) \rightarrow \Lambda$ given by $str(\begin{pmatrix}
A & B \\
C & D
\end{pmatrix})=tr(A)-tr(D)$, where for a matrix $X$, $tr(X)$ denotes the sum of the diagonal entries of $X$. We note that the supertrace function is $\Lambda_0$-linear and satisfies $str(ST)=str(TS)$ for $S,T \in M_{k,l}(\Lambda)$.

The general linear supergroup ${\rm GL}(U)$ is the set
${\rm GL}(U) = \{g \in \Endo_{\Lambda} (U )_0 : g \,\, is \,\,  invertible \}$.


 \subsection{}\label{action} Let $V$ be a super vector space over $\mathbb C$ and let $V^{\otimes k}=V \otimes V \otimes \cdots \otimes V$ be the $k$-fold tensor product of $V$. Then $V^{\otimes k}$ is a $\mathfrak{gl}(V)$-module by defining 
\begin{multline}\label{e:action} x.(v_1 \otimes v_2 \otimes \cdots \otimes v_k)= x.v_1 \otimes v_2 \otimes \cdots \otimes v_k +
  (-1)^{|x||v_1|}v_1 \otimes x.v_2 \otimes \cdots \otimes v_k+ \\ \cdots +(-1)^{|x|(|v_1|+|v_2|+\cdots +|v_{k-1}|)}v_1 \otimes v_2 \otimes \cdots \otimes x.v_k,
  \end{multline}
   where $x \in \mathfrak{gl}(V)$ and $v_j \in V$ are $\mathbb Z_2$-homogeneous elements of degrees $|x|$ and $|v_j|$ respectively.

   When $U:=V \otimes_{\mathbb C} \Lambda$, the general linear superalgebra $\mathfrak{gl}(U)$ acts on $U^{\otimes k}$ by the ordinary derivation, i.e., $$x.(u_1 \otimes u_2 \otimes \cdots \otimes u_k)= \sum_{i=1}^k u_1 \otimes u_2 \otimes \cdots \otimes x.u_i \otimes \cdots \otimes u_k.$$

 On the other hand the symmetric group $S_k$ acts on $V^{\otimes k}$ by
 \begin{equation} (i,i+1). v_1 \otimes v_2 \otimes \cdots \otimes v_i \otimes v_{i+1}\otimes  \cdots \otimes v_k=(-1)^{|v_i||v_{i+1}|}v_1 \otimes v_2 \otimes \cdots \otimes v_{i+1} \otimes v_i\otimes  \cdots \otimes v_k,
 \end{equation} 
 where $(i, j)$ denotes a transposition in  $S_k$ and $v_i, v_{i+1}$ are $\mathbb Z_2$-homogeneous. More generally the action of $S_k$ on $V^{\otimes k}$ is defined as follows: For $\sigma \in S_k$ and $\underline{v}=v_1 \otimes v_2 \otimes \cdots \otimes v_k$,
 $$\sigma .\underline{v}=\gamma(\underline{v}, \sigma^{-1})(v_{\sigma^{-1}(1)}\otimes v_{\sigma^{-1}(2)} \otimes \cdots \otimes v_{\sigma^{-1}(k)}),$$ where $\gamma(\underline{v}, \sigma)=\prod_{(i,j) \in Inv(\sigma)}(-1)^{|v_i||v_j|}$, with $Inv(\sigma)=\{(i,j): i < j \,\, \text{and} \,\, \sigma(i) > \sigma(j)\}$. 

 It is easy to see that $\gamma(\underline{v},\sigma\tau)=\gamma(\sigma^{-1}\underline{v},\tau)\gamma(\underline{v},\sigma)$ for two permutations $\sigma$ and $\tau$.

 We denote the representations of $S_k$ and $\mathfrak{gl}(V)$ on $V^{\otimes k}$ by $\eta_k$ and $\rho_k$ respectively.

 \subsection{Isomorphisms}\label{iso} In this subsection we list out all the isomorphisms we require in this paper. Let $V$ be a super vector space over $\mathbb C$.

 1. $V \otimes V^* \cong \Endo(V)$ defined by $(v,\phi) \mapsto T_{(v,\phi)}$ where $T_{(v,\phi)}(w)=v\phi(w)$.

 This implies that $\Endo(V)^{\otimes k} \cong (V \otimes V^*)^{\otimes k}$ for $k \in \mathbb Z_{\geq 0}$. 

 2. For two super vector spaces $V$ and $W$ over $\mathbb C$, we have an isomorphism $V \otimes W \rightarrow W \otimes V$ defined by $v\otimes w \mapsto (-1)^{|v||w|}w \otimes v$, where $v$ and $w$ are homogeneous elements of $V$ and $W$ respectively. 

 3. We have an evaluation map $(V^*)^{\otimes k} \times V^{\otimes k} \rightarrow \mathbb C$ defined by $$(f_1^*\otimes f_2^*\otimes \cdots \otimes f_k^*, v_1\otimes v_2\otimes \cdots \otimes v_k) \mapsto (-1)^{p(\underline{f}^*,\underline{v})}f_1^*(v_1)\cdots f_k^*(v_k),$$ where $\underline{f}^*=f_1^*\otimes f_2^*\otimes \cdots \otimes f_k^*$,  $\underline{v}=v_1\otimes v_2\otimes \cdots \otimes v_k$ for homogeneous elements $f_i^*$ and $v_i$, and $p(\underline{f}^*,\underline{v})$ is defined as follows. For homogeneous elements $f_i^*$ and $v_i$, let $d_j=|v_j|(|f^*_{j+1}|+\cdots +|f^*_k|)$ with $d_k=0$. Then $p(\underline{f}^*,\underline{v})=\sum_{j=1}^k d_j$ and extend this definition  by linearity.  

 4. The above evaluation map is a non-degenerate bilinear form and hence we get an isomorphism between $(V^*)^{\otimes k}$ and $(V^{\otimes k})^*$. 
 
 5. Let $\tau$ be the permutation which takes $(1,2,\cdots, k,k+1, \cdots ,2k)$ to $(\tau(1), \tau(2), \cdots ,\tau(k),\tau(k+1), \cdots ,\tau(2k))=(1,3,5, \cdots ,2k-1, 2,4,6, \cdots ,2k)$. The symmetric group $S_{2k}$ acts on $(V\dsum V^*)^{\tensor k}$. Under this action, the permutation $\tau$ takes $V^{\otimes k} \otimes (V^*)^{\otimes k}$ to $(V \otimes V^*)^{\otimes k}=(V \otimes V^*) \otimes (V \otimes V^*) \otimes \cdots \otimes (V \otimes V^*)$. Thus we get an $\Endo(V)$-module isomorphism between $V^{\otimes k} \otimes (V^*)^{\otimes k}$ and $(V \otimes V^*)^{\otimes k}$. We then have 
\[ V^{\otimes k} \otimes (V^*)^{\otimes k}\cong (V \otimes V^*)^{\otimes k} \cong\Endo(V)^{\otimes k} .\]

Let $\underline{v}=v_1 \otimes v_2 \otimes \cdots \otimes v_k$ and $\underline{f}^*=f^*_1\otimes f^*_2\otimes \cdots \otimes f^*_k$. The map in the reverse direction is defined by $\underline{v} \otimes \underline{f}^* \mapsto (-1)^{p(\underline{v},\underline{f}^*)}(v_1 \otimes f_1^*) \otimes (v_2 \otimes f_2^*) \otimes \cdots \otimes (v_k \otimes f_k^*)$. Note that the sign $(-1)^{p(\underline{v},\underline{f}^*)}$ is exactly $\gamma(\underline{v}\tensor\underline{f}^*, \tau^{-1})$.

6. Let $W=W_1\oplus \cdots \oplus W_k$, where $W_i$ are super vector spaces over $\mathbb C$. Then the group $S_{2k}$ acts on $(W\oplus W^*)^{\tensor 2k}$. The permutation $\tau$ defined above then induces an isomorphism between $W_1^*\tensor \cdots \tensor W_k^* \tensor W_1\tensor \cdots \tensor W_k$ and $W_1^*\tensor W_1\tensor \cdots \tensor W_k^* \tensor W_k$. This isomorphism followed by the evaluation map to $\C$ gives a non-degenerate pairing between $W_1^*\tensor \cdots \tensor W_k^* $ and $(W_1\tensor \cdots \tensor W_k)^*$. As the $S_{2k}$ action commutes with the $\Endo(W)$-action, so the isomorphism $W_1^*\tensor \cdots \tensor W_k^* \cong (W_1\tensor \cdots \tensor W_k)^*$ is an $\Endo(W)$ isomorphism. 

It may be noted here that under the above isomorphism, the dual basis element $T_{w_1\tensor \cdots\tensor w_n}\in (W_1\tensor \cdots \tensor W_k)^* $ corresponds to $(-1)^{p(\underline{w},\underline{w})}T_{w_1}\tensor \cdots\tensor T_{w_k}\in W_1^*\tensor \cdots \tensor W_k^* $ where $\underline{w}=w_1\tensor \ldots\tensor w_k$.

7. Let $\nu\in S_{2k}$ be the permutation $(1\, 2)(3\, 4)\cdots (2k-1\, 2k)$ and $\text{ev}^{\tensor k}:(V^*\tensor V)^{\tensor k} \to \C$ denote the evaluation map $(f_1^* \otimes v_1) \otimes (f_2^* \otimes v_2) \otimes \cdots \otimes (f_k^* \otimes v_k)\mapsto f_1^*(v_1)\cdots f_k^*(v_k)$. We have a non-degenerate bilinear form $\Endo(V^{\tensor k})\tensor V^{\tensor k}\tensor (V^*)^{\tensor k}\to \C$ given by $\langle A,\underline{v}\tensor \underline{f}^*\rangle:= \text{ev}^{\tensor k}[(\nu.\tau).(A(\underline{v})\tensor \underline{f}^*)]$. This gives an $\Endo(V)$ isomorphism \[\Theta: \Endo(V^{\tensor k})\cong (V^{\tensor k}\tensor (V^*)^{\tensor k})^*\]

\begin{remark} All the isomorphisms given above also hold true with $V$ replaced by $U:=V \otimes_{\mathbb C} \Lambda$.
\end{remark}

\subsection{Supertrace} For $U=V \otimes_{\mathbb C} \Lambda$, we have $U^*={\rm Hom}_{\Lambda}(U,\Lambda)$ and $U \otimes U^* \cong \Endo_{\Lambda}(U)$ as graded $\Lambda$-algebras where the multiplication on $U \otimes U^*$ is defined by $(v \otimes \phi)(w \otimes \psi)=v \otimes \phi(w)\psi$. 

We define the supertrace map $str:\Endo_{\Lambda}(U) \rightarrow \Lambda$ by $$str(v \otimes \phi)=(-1)^{|v||\phi|}\phi(v)$$ and with the above identification this definition agrees with the definition given on $\Endo_{\Lambda}(U)$. It satisfies $str(AB)=(-1)^{|A||B|}str(BA)$, for $A,B \in \Endo_{\Lambda}(U)$. When restricted to $\Endo_{\Lambda}(U)_0$, the $str$ map behaves like the ordinary trace map in the sense that $str(AB)=str(BA)$ for $A,B \in \Endo_{\Lambda}(U)_0$ and $str(CDC^{-1})=str(D)$ for $D \in \Endo_{\Lambda}(U)$ and $C \in GL(U)$.

 \subsection{Schur-Weyl duality}
For a super vector space $V$ over $\mathbb C$, the actions of $\mathfrak{gl}(V)$ on $V^{\otimes k}$ and the right action of $S_k$ on $V^{\otimes k}$  defined in 
  \S \ref{action} commute with each other and the super analogue of Schur-Weyl duality says that the span of the images of $S_k$ and $\mathfrak{gl}(V)$ in $\mathrm{End}(V^{\otimes k})$ are centralizers of each other (see \cite{BR}). In particular we have a surjective  algebra homomorphism: \[\eta_k: \mathbb C[S_k] \rightarrow \mathrm{End}_{\mathfrak{gl}(V)}(V^{\otimes k}),\] where $\mathrm{End}_{\mathfrak{gl}(V)}(V^{\otimes k})$ is the space of $\mathfrak{gl}(V)$-module homomorphisms of $V^{\otimes k}$. This is equivalent to taking the right action of $S_k$ on $V^{\otimes k}$ and the induced homomorphism \[\Psi_k: \mathbb C[S_k] \rightarrow \mathrm{End}_{\mathfrak{gl}(V)}(V^{\otimes k})^{op}.\]

 \subsection{Schur-Weyl duality over $\Lambda$} Let $V$ be a super vector space over $\mathbb C$ and let $U=V \otimes_{\mathbb C} \Lambda$. Let $U^{\otimes r}=U \otimes \cdots \otimes U$ ($r$-times) be the $r$-fold tensor product of $U$. The general linear supergroup ${\rm GL}(U)$ acts on $U^{\otimes r}$ by $$g.(w_1 \otimes w_2 \otimes \cdots \otimes w_r)=g.w_1 \otimes g.w_2 \otimes \cdots \otimes g.w_r$$ for $g \in {\rm GL}(U)$. The dual superspace $U^*$ of $U$ has a natural $GL(U)$-module structure given by $g.f(v)=f(g^{-1}.v)$, for all $v \in U$. 

On the other hand, the symmetric group $S_r$ acts on $U^{\otimes r}$ by
 \begin{equation} (i,i+1). v_1 \otimes v_2 \otimes \cdots \otimes v_i \otimes v_{i+1}\otimes  \cdots \otimes v_r=(-1)^{|v_i||v_{i+1}|}v_1 \otimes v_2 \otimes \cdots \otimes v_{i+1} \otimes v_i\otimes  \cdots \otimes v_r,
 \end{equation} 
 where $(i, j)$ denotes a transposition in  $S_r$ and $v_i, v_{i+1}$ are $\mathbb Z_2$-homogeneous. We denote the representations of $S_r$ and ${\rm GL}(U)$ on $U^{\otimes r}$ by $\eta_r$ and $\rho_r$ respectively. 
 
 The group ring $\Lambda S_r = \mathbb C S_r \otimes \Lambda$ is an associative superalgebra, with $\mathbb C S_r$
regarded as purely even. By extending the action of $S_r$ on $U^{\otimes r}$ $\Lambda$-linearly we get an action of the super algebra $\Lambda S_r$ on $U^{\otimes r}$. It is easy to see that the actions of $\Lambda S_r$ and $GL(U)$ on $U^{\otimes r}$ commute with each other.

We then have the following super analogue of Schur-Weyl duality and for a proof see \cite{LZ}.

\begin{theorem}
Let $End_{(GL(U)}(U^{\otimes r})=\{\phi \in End_{\Lambda}(U^{\otimes r}): g\phi=\phi g, \forall g \in GL(U)\}$. Then $End_{GL(U)}(U)^{\otimes r}=\eta (\Lambda S_r)$.
\end{theorem}

\section{Tensor FFT for $\mathfrak{gl}(V)$} 

\subsection{Tensor FFT over $\mathbb C$} Let $V=V_0 \oplus V_1$ be a super vector space over $\mathbb C$. The tensor version of the first fundamental theorem of invariant theory for $\mathfrak{gl}(V)$ describes a spanning set for $(V^{\otimes k} \otimes (V^*)^{\otimes k})^{\mathfrak{gl}(V)}$ which can be derived from the super analogue of the Schur-Weyl duality described above. Here we provide a proof for the convenience of the reader. 
 
 Let $e_1, \cdots ,e_{m+n}$ be a basis for $V$ with the first $m$ vectors are of degree $0$ and the rest are of degree $1$. Let $e_1^*, \cdots ,e_{m+n}^*$ be the dual basis for $V^*$. 
 
We set $S=S_0 \cup S_1$, where $S_0 = \{1, 2, \cdots, m\}$ and $S_1 = \{m + 1,m + 2, \cdots,m + n\}$. We define parity for $I$ as $p(i) = 0$ if $i \in S_0$ and $p(i) = 1$ if $i \in S_1$.  We also set $|E_{ij}|=p(i)+p(j)$. 

 For $\sigma \in S_k$, we define 
 \[\theta_{\sigma}=  \sum_I (-1)^{p(I,I)} \gamma(I,\sigma^{-1})(e_{i_{\sigma^{-1}(1)}} \otimes \cdots \otimes e_{i_{\sigma^{-1}(k)}}) \otimes (e_{i_1}^* \otimes \cdots \otimes e_{i_k}^*),\]
 where the sum is over all the $k$-element subsets $I=\{i_1, i_2, \cdots ,i_k\}$ of $S$ and $\gamma(I,\sigma)=\gamma(e_I,\sigma)$ and $p(I,I)=p(e_I, e_I)$ with $e_I=e_{i_1} \otimes \cdots \otimes e_{i_k}$. 
 
 In what follows we use the notation $p(I,J)=p(e_I,e_J)$, where $e_I=e_{i_1} \otimes \cdots \otimes e_{i_k}$ and $e_J=e_{j_1} \otimes \cdots \otimes e_{j_k}$ for subsets $I=\{i_1, i_2, \cdots ,i_k\}$ and $J=\{j_1, j_2, \cdots ,j_k\}$. 
 
 Note that the elements $(\theta_{\sigma})_{\sigma \in S_k}$ are in $(V^{\otimes k} \otimes (V^*)^{\otimes k})^{\mathfrak{gl}(V)}$. 

\begin{theorem}\label{fft-gln} (Tensor FFT for $\mathfrak{gl}(V)$). The space $(V^{\otimes k} \otimes (V^*)^{\otimes k})^{\mathfrak{gl}(V)}$ is spanned by the set $\{\theta_{\sigma}:\sigma \in S_k\}$.
\end{theorem}

\begin{proof}
Since $V \otimes V^* \cong \mathrm{End}(V) \cong \mathfrak{gl}(V)$, we have an isomorphism \[\Phi: (V^{\otimes k} \otimes (V^*)^{\otimes k}) \rightarrow \mathrm{End}(V^{\otimes k}).\] 

Since the tensor product action of $\mathfrak{gl}(V)$ on $V \otimes V^*$ corresponds to the adjoint action on $\mathfrak{gl}(V)$, we have an isomorphism of $\mathfrak{gl}(V)$-invariants: \[\Phi: (V^{\otimes k} \otimes (V^*)^{\otimes k})^{\mathfrak{gl}(V)} \rightarrow (\mathrm{End}(V^{\otimes k}))^{\mathfrak{gl}(V)} \cong \mathrm{End}_{\mathfrak{gl}(V)}(V^{\otimes k}),\]

where $\mathrm{End}_{\mathfrak{gl}(V)}(V^{\otimes k})$ denotes the space of $\mathfrak{gl}(V)$-module endomorphisms of $V^{\otimes k}$. 

By Schur-Weyl duality, we have $\eta_k(\mathbb C[S_k]) = \mathrm{End}_{\mathfrak{gl}(V)}(V^{\otimes k})$. So we have an isomorphism (we still denote it by $\Phi$):
\[\Phi: (V^{\otimes k} \otimes (V^*)^{\otimes k})^{\mathfrak{gl}(V)} \rightarrow \eta_k(\mathbb C[S_k]).\] 

We claim that the map $\Phi$ maps $\theta_{\sigma}$ to $\eta_k(\sigma)$. Let $J=\{j_1, j_2, \cdots ,j_k\}$. Then $$\Phi(\theta_{\sigma})(e_{j_1} \otimes \cdots \otimes e_{j_k})$$ 

$$=\sum_I (-1)^{p(I,I)} \gamma(I,\sigma^{-1})(e_{i_{\sigma^{-1}(1)}} \otimes \cdots \otimes e_{i_{\sigma^{-1}(k)}}) \otimes (e_{i_1}^* \otimes \cdots \otimes e_{i_k}^*)(e_{j_1} \otimes \cdots \otimes e_{j_k})$$

$$=\sum_I (-1)^{(p(I,J)+p(I,I))}\gamma(I,\sigma^{-1}) (e_{i_1}^*(e_{j_1})\cdots e_{i_k}^*(e_{j_k}))(e_{i_{\sigma^{-1}(1)}} \otimes \cdots \otimes e_{i_{\sigma^{-1}(k)}}).$$
$$= \gamma(J, \sigma^{-1})(e_{j_{\sigma^{-1}(1)}} \otimes \cdots \otimes e_{j_{\sigma^{-1}(k)}})=\eta_k(\sigma)(e_{j_1} \otimes \cdots \otimes e_{j_k}).$$

So the set $\{\theta_{\sigma}:\sigma \in S_k\}$ spans $(V^{\otimes k} \otimes (V^*)^{\otimes k})^{\mathfrak{gl}(V)}$.

\end{proof}

\subsection{Tensor FFT over $\Lambda$}\label{s.fft.gln} Let $V$ be a super vector space over $\mathbb C$ and let $U=V \otimes_{\mathbb C} \Lambda$. Then  $U^{\otimes r} \otimes_{\Lambda} U^{*\otimes s}$ is a $GL(U)$ module. For $\alpha \in \mathbb C^*$, the element $\alpha I$ belongs to the centre of $GL(U)$ and it acts on $U^{\otimes r} \otimes U^{*\otimes s}$ by $\alpha I.x=\alpha^{r-s}x$, where $I$ is the identity matrix of size $(m+n) \times (m+n)$. Hence there are no invariants if $r \neq s$, so we only need to consider the $GL(U)$-invariants of $U^{\otimes r} \otimes_{\Lambda} U^{*\otimes r}$. Then we have a $GL(U)$-module isomorphism $$U^{\otimes r} \otimes_{\Lambda} U^{*\otimes r} \cong \Endo_{\Lambda}(U^{\otimes r}).$$
From the Schur-Weyl duality it follows that $$(U^{\otimes r} \otimes_{\Lambda} U^{*\otimes r})^{GL(U)} \cong \Endo_{GL(U)}(U^{\otimes r}) \cong \eta (\Lambda S_r).$$

\noindent We denote the image of $\sigma \in S_r$ in $(U^{\tensor r}\tensor (U^*)^{\tensor r})^{GL(U)}$ also by $\theta_{\sigma}$ as in the above theorem. 

\section{Graded invariants for mixed tensor spaces}
\newcommand{\iso}{\cong}
For a complex super vector space $V$ of super dimension $(m|n)$ and non-negative integers $t, b$, we denote by $V_b^t$ the tensor space $V^{\tensor b}\tensor (V^{*})^{\tensor t}$. Recall that the tensor space has a natural $\Z_2$-grading where the degree of a homogeneous element $v_1\tensor \cdots \tensor v_b\tensor\phi_1\tensor \cdots \tensor\phi_t$ is the sum of the degrees of the individual elements taken modulo $2$. By a {\it mixed tensor space} we mean the direct sum  of finitely many tensor spaces, \ie of the form $\dsum_{i=1}^s (V)_{b_i}^{t_i}$,  for an $s\in \N$ and non-negative integers $t_i, b_i$, $i=1,\ldots s$. The mixed tensor space inherits its $\Z_2$-grading from its graded direct summands. Let $U$ denote the free $\Lambda$-module $ V\tensor_\C\Lambda$, introduced earlier. Recall that there is a natural identification $V\hookrightarrow U$, such that $v\mapsto v\tensor 1$. As was also noted there, the supergroup $\gl(U)$ acts on $U$, hence it acts naturally on the tensor space $U^t_b$ and acts diagonally on the mixed tensor space  $\dsum_{i=1}^s (U)_{b_i}^{t_i}\cong(\dsum_{i=1}^s (V)_{b_i}^{t_i})\tensor_\C \Lambda$.

\subsection{Supersymmetric algebra of the dual of the mixed tensor space} For the $\Z_2$-graded vector space $V$ of superdimension $(m|n)$ over $\C$, let $\{e_1,\ldots,e_{n+m}\}$ be a basis with the first $m$ vectors from the even part and the next $n$ from the odd part. For notational convenience, we rename the dual basis $e_1^*,\ldots, e_{m+n}^*$ by $x_1,\ldots, x_m, y_1\ldots, y_n$ respectively. By abuse of notation, we denote by $\{e_1,\ldots ,e_{n+m}\}$  and $\{x_1,\ldots, x_m, y_1\ldots, y_n\}$, respectively the corresponding $\Lambda$-basis of $U$ and $U^*$ respectively, under the identification $V\hookrightarrow U$ and $V^*\hookrightarrow V^*\tensor_\C\Lambda\iso U^*$.  The tensor algebra on $U^*$ is $T(U^*)= \oplus_n T^n(U^*)$ where $T^n(U^*):= (U^*)^{\tensor_\Lambda n}$ for $n>0$ and $T^0(U^*):=\Lambda$. This is a $\Z$-graded superalgebra. Let $I(U^*)$ be the ideal generated by  elements of the form $u\tensor v- (-1)^{|u||v|}v\tensor u$ where $u,v\in U^*$. Note that $I(U^*)$ is both a $\Z$-graded and a $\Z_2$-graded ideal in $T(U^*)$. Hence the supersymmetric algebra $S(U^*):=T(U^*)/I(U^*)$ inherits a $\Z$-grading as well as a  $\Z_2$-grading from $T(U^*)$. We denote the image of $T^r(U^*)$ in $S(U^*)$ by $S^r(U^*)$ and the natural map from $T^r(U^*)\to S^r(U^*)$ is denoted as $\varpi_r$. In $S(U^*)$, we have the relations $x_i\tensor x_j=x_j\tensor x_i$, $x_i\tensor y_j=y_j\tensor x_i$ and $y_i\tensor y_j=-y_j\tensor y_i$.  Hence we can identify $S(V^*)$ with $(\C[x_1,\ldots, x_m]\tensor_\C \Lambda)\tensor_\Lambda(\wedge[y_1,\ldots,y_n]\tensor_\C\Lambda)$ where $\C[x_1,\ldots, x_m]$ denotes the polynomial algebra over $m$ commuting variables and $\wedge[y_1\ldots,y_n]$ denotes the exterior algebra generated by $y_1, \ldots, y_n$ satisfying the anti-commutativity relations $y_iy_j=-y_jy_i$. In other words, $S(U^*)$ has a $\Lambda$-basis given by monomials of the form $x_1^{r_1}\cdots x_m^{r_m}y_1^{s_1}\cdots y_m^{s_m}$ where $r_i\in \N\cup\{0\}$ and $s_i\in \{0,1\}$. One may identify 
$S^r(U^*)$ with the $GL(U)$-stable subspace $e(r)T^r(U^*)$ where $e(r):=\frac{1}{r!}\sum_{\sigma\in\Sym_r}\sigma$. Since the $GL(U)$-action on $T^r(U^*)$ commutes with the $\Sym_r$-action, so the action of $GL(U)$ on $T^r(U^*)$ descends to $S^r(U^*)$, thereby making $\varpi_r$ a $\gl(U)$-equivariant map.

In the case of $V_b^t$, 
let $T_{i_1\cdots i_b}^{j_1\cdots j_t}$ denote the linear functional on $V_b^t$ that takes value $1$ only on the basis vector $e_{i_1}\tensor \cdots \tensor e_{i_b}\tensor e^*_{j_1}\tensor \cdots \tensor e^*_{j_t}$ and $0$ on others. Under the natural identification, $(U_b^t)^*\iso (V_b^t)^*\tensor \Lambda$, we denote the $\Lambda$-linear map corresponding to $T_{i_1\cdots i_b}^{j_1\cdots j_t}$ also by the same notation. With this notation, following the same reasoning as in the case of $S(U^*)$, the super ring $S((U_b^t)^*)$ can be identified with the super ring of polynomials in the variables $T_{i_1\cdots i_b}^{j_1\cdots j_t}$ with the supercommutativity relation $T_{i_1\cdots i_b}^{j_1\cdots j_t}T_{i_1'\cdots i_b'}^{j_1'\cdots j_t'}=(-1)^{(\sum _{k}a_{i_k}+\sum_{l} a_{j_l})(\sum _{k}a_{i_k'}+\sum_{l} a_{j_l'})}T_{i_1'\cdots i_b'}^{j_1'\cdots j_t'}T_{i_1\cdots i_b}^{j_1\cdots j_t}$ where $a_{i}$ denotes the parity of the basis element $e_i$. More generally, for $W=\dsum_{i=1}^s U_{b_i}^{t_i}$ the mixed tensor space as defined above we use the natural identification of $W^*$ with $\dsum_{i=1}^s(U_{b_i}^{t_i})^*$ to obtain a $\Lambda$-basis for $W^*$ given by the linear maps $T(i)_{l_1\ldots l_{b_i}}^{u_1\ldots, u_{t_1}}\in (U_{b_i}^{t_i})^*$, for $i=1,\ldots, s$,  $1\leq l_j\leq m+n$ and $1\leq u_j\leq m+n$. Under this identification the linear map $T(i)_{l_1\ldots l_{b_i}}^{u_1\ldots, u_{t_i}}\in W^*$ evaluates on $w=(w_1,\ldots,w_s)$ to $T(i)_{l_1\ldots l_{b_i}}^{u_1\ldots u_{t_i}}(w_i)$. With this notation, the supersymmetric algebra $S(W^*)$ can be identified with the superalgebra of polynomials in the variables $T(l)_{i_1\cdots i_b}^{j_1\cdots j_t}$, $l=1,\ldots, s$ with the supercommutativity relation $T(l)_{i_1\cdots i_{b_l}}^{j_1\cdots j_{t_l}}T(l')_{i_1'\cdots i_{b_l}'}^{j_1'\cdots j_{t_l}'}=(-1)^{(\sum _{k}a_{i_k}+\sum_{l} a_{i_l})(\sum _{k}a_{i_k'}+\sum_{l} a_{i_l'})}T(l')_{i_1'\cdots i_{b_l}'}^{j_1'\cdots j_{t_l}'}T(l)_{i_1\cdots i_{b_l}}^{j_1\cdots j_{t_l}}$ where $a_{i}$ denotes the parity of the basis element $e_i$. From the above identifications of $S(W^*)$ and $S((U_{b_i}^{t_i})^*)$ as  polynomial algebras, it follows that there is a natural isomorphism from $\tensor_{i=1}^sS((U_{b_i}^{t_i})^*)$ to $S(W^*)$ given by $(\phi_1\tensor \cdots\tensor\phi_s)\mapsto \phi_1\cdots\phi_s$. 
\subsection{Graded picture invariants on $W$}\label{s:picture}
\newcommand{\ts}[1]{\psi_{#1}}
Let $W=\dsum_{i=1}^s U_{b_i}^{t_i}$ for an $s\in \N$ and non-negative integers $t_i, b_i$, $i=1,\ldots s$. The $\gl(U)$ action on $W$ induces an action on $T^r(W^*)$ for each $r\in\Z^{\geq 0}$. This action commutes with action of the symmetric group $\Sym_r$ for each $r>0$ hence the $\gl(U)$-action descends to the subspace $e(r)T^r(W^*)$ which in turn can be identified with $S^r(W^*)$. Thus, we note that the $\gl(U)$ action on $W$ descends to the associated supersymmetric algebra $S(W^*)$ of ``polynomials" on $W$. With respect to this action of $\gl(U)$ on $S(W^*)$ we are interested in finding a generating set for the invariant subalgebra $S(W^*)^{\gl(U)}$. For this, we define the notion of a {\it graded picture invariant} following \cite{dks}. In the supersymmetric setting, to an element of the symmetric group, we associate a polynomial in $S(W^*)$ which we shall call as a graded picture invariant. A {\it graded picture invariant} is defined as follows: given an $s$-tuple $(m_1,\ldots,m_s)$ of non-negative integers such that $\sum_{k=1}^s m_kt_k=\sum_{k=1}^s m_k b_k=N$ and a $\sigma\in\Sym_N$ we associate the polynomial $\ts{\sigma}\in S(W^*)$ given by \begin{equation}\label{eq:picture}
   \sum_{r_1,\ldots, r_N\in \{1,\ldots, n\}}(-1)^{p_\sigma(r_1,\ldots, r_N, m_1,\ldots, m_s)}\prod_{i=1}^{s}\prod_{j=1}^{m_i}T(i)_{r_{\sigma(\sum_{p<i}m_pb_p+(j-1)b_i+1)}\cdots r_{\sigma(\sum_{p<i}m_pb_p+jb_i)}}^{r_{(\sum_{p<i}m_pt_p+(j-1)t_i+1})\cdots r_{(\sum_{p<i}m_pt_p+jt_i)}}  
\end{equation}where $p_\sigma(I,M)$ for an $N$-tuple $I=(r_1,\ldots,r_N)$ of $N$ elements from $\{1,\ldots  n\}$ and $M=(m_1,\ldots, m_s)\in (\N\cup \{0\})^s$ such that $\sum_k m_{k}t_k=\sum_k m_k b_k=N$ takes the value $\pm 1$ determined by the formula 
$\gamma(I,\sigma^{-1})\times (-1)^{q(I,M,\sigma)}$
 where $q(I,M,\sigma)$ is
\begin{multline*}
\sum_{i=1}^s\sum_{j=1}^{m_i}\left[\left(\sum_{l=1}^{b_i}a_{\sigma(\sum_{p<i}m_pb_p+(j-1)b_i+l)}\right)\left(\sum_{k=1}^{\sum_{p<i} m_pt_p+(j-1)t_i}a_{k}\right)\right]+\\
\sum_{i=1}^N a_i+\sum_{i=1}^{N-1}a_{i}\left(\sum_{j=i+1}^N a_{j}\right)+p(\underline{w},\underline{w}).
\end{multline*}
\noindent
In the above formula, $a_i$ is the parity of $e_{r_i}$ (and also of its dual vector $e_{r_i}^*$)
and $\underline{w}:=w_1 \otimes w_2 \otimes \cdots \otimes w_{\sum_{i=1}^sm_i}$ where each $w_{\sum_{p<i}m_p+j}:=e_{r_{\sum_{p<i}m_pb_p+(j-1)b_i+1}}\tensor\cdots\tensor e_{r_{\sum_{p<i}m_pb_p+jb_i}}\tensor e_{r_{\sigma^{-1}(\sum_{p<i}m_pt_p+(j-1)t_i+1)}}^*\tensor\cdots\tensor e_{r_{\sigma^{-1}(\sum_{p<i}m_pt_p+jt_i)}}^*$ for $i=1,\ldots,s; ~j=1,\ldots ,m_i$. It can be seen that the degree of $w_{\sum_{p<i}m_p+j}$ is $$\sum_{k=1}^{b_i}a_{(\sum_{p<i}m_pb_p+(j-1)b_i+k}+\sum_{k=1}^{t_i}a_{\sigma^{-1}(\sum_{p<i}m_pt_p+(j-1)t_i+k)}.$$

Following the convention that for $\sigma\in S_{N}$, $\sigma.(r_1,\ldots, r_{N}):=(r_{\sigma^{-1}(1)},\ldots, r_{\sigma^{-1}(N)})$, it may be seen that $\gamma(\sigma. I,\tau^{-1})\gamma(I,\sigma^{-1})=\gamma(I,(\tau\sigma)^{-1})$ (see \S \ref{action}). Using this, we note that the above sign $(-1)^{q(I,M,\sigma)}\gamma(I,\sigma^{-1})$  may be realised as $\gamma(\mu^{-1}.(I,\sigma I),(\nu\tau\hat{\sigma}\mu)^{-1})$ where for a $\sigma\in S_N$ we define $\hat{\sigma}\in S_{2N}$ as 
\begin{align*}
        \hat{\sigma}(i)&=\sigma(i) \quad\textup{ for }i\leq N \\ \hat{\sigma}(i)&=i\quad\textup{ for }i>N
\end{align*}
and for an $M=(m_1,\ldots, m_s)\in (\N\cup \{0\})^s$ such that $\sum_{k=1}^s m_{k}t_k=\sum_{k=1}^s m_k b_k=N$ the permutation $\mu$ takes $(1,\ldots, 2N)$ to $(1,\ldots b_1, N+1, \ldots , N+t_1, b_1+1, \ldots 2b_1, \cdots )$. 


The next proposition establishes that the above graded picture invariants generate the invariant super ring $S(W^*)^{\gl(U)}$. This is analogous to the statement of \cite[Proposition~7]{dks}. Further, it may be remarked that the proof follows the same line of argument as in \cite[Proposition~7]{dks}. However, the identifications involved in the proof there will have to be modified to the braided category of $\Z_2$-graded vector spaces. This leads to polynomials which are different from that of {\it loc.cit}. 
\begin{theorem}\label{t:picture}
  Given a mixed tensor space $W=\oplus_{i=1}^s U_{b_i}^{t_i}$, the graded picture invariants as defined above $\Lambda$-linearly span $S(W^*)^{\gl(U)}$.
\end{theorem}
\begin{proof}
The degree $n$ component of the ring $S(W^*)$, denoted as $S^n(W^*)$,  is the image of $T^n(W^*)$ under $\varpi_n$.  The basis for $S^n(W^*)$ is then given by the degree $n$ monomials in the super-commuting variables $T(i)_{l_1\cdots l_{b_i}}^{u_1\cdots u_{t_i}}$ arranged in order from $i=1$ to $i=s$ and a fixed ordering of the standard basis of $U_{b_i}^{t_i}$ for each $i$. Hence it can be seen that $$S^n(W^*)\iso \oplus_{\substack{m_1,\ldots, m_s\\ m_1+\cdots+m_s=n}}(\tensor_{i=1}^sS^{m_i}((U_{b_i}^{t_i})^*)).$$ This isomorphism may be structurally described as coming from the degree $n$ component of $S(W^*)$ being expressed as a direct sum of its  $(m_1,\ldots, m_s)$-multidegree summands obtained by setting the multidegree of the basis element $T(i)_{l_1 \ldots l_{b_i}}^{u_1,\ldots, u_{t_i}}$ to be the $s$-tuple $(0,\ldots,0, 1,0, \ldots)$ with $1$ only in the $i^\textit{th}$-coordinate.
We  have the natural  surjection from $\varpi_{m_i}:((U_{b_i}^{t_i})^*)^{\tensor m_i} \to S^{m_i}((U_{b_i}^{t_i})^*)$ for each $i$. Since the $\gl(U)$-action on $((U_{b_i}^{t_i})^*)^{\tensor m_i}$ commutes with the $\Sym_{m_i}$-action and $S^{m_i}((U^{t_i}_{b_i})^*)$ can be identified under $\varpi_{m_i}$ with the subspace $e(m_i)\left(((U_{b_i}^{t_i})^*)^{\tensor m_i}\right)$ where $e(m_i):=\frac{1}{m_i!}\sum_{\sigma\in\Sym_{m_i}}\sigma$, so the $\gl(U)$-action descends to $S^{m_i}((U_{b_i}^{t_i})^*)$ making $\varpi_{m_i}$ to be $\gl(U)$-equivariant. Thereby, we get a $\gl(U)$-equivariant surjection
$\varpi_{m_1}\tensor\ldots\tensor \varpi_{m_s}:\tensor_{i=1}^s((U_{b_i}^{t_i})^*)^{\tensor m_i} \twoheadrightarrow \tensor_{i=1}^sS^{m_i}((U_{b_i}^{t_i})^*)$. 
Set $W_{\sum_{j<r} m_j+i}:=U_{b_r}^{t_r}$ for $i=1,\ldots ,m_r$. Then,
by the isomorphism listed in (6) of \S \ref{iso} we have,
\begin{equation}\label{eq:1}(\tensor_{i=1}^kW_i)^*\iso \tensor_{i=1}^k(W_i^*)\end{equation} where $k=\sum_i m_i$, given by sending the basis element $T_{w_1\tensor \cdots \tensor w_{\sum_i m_i}}\mapsto (-1)^{p(\underline{w},\underline{w})}T(1)_{w_1}\tensor\cdots\tensor T(s)_{w_{\sum_i m_i}}$.  Here $w_i\in W_i$ denotes a basis vector of $W_i$ and if $W_i=U_{b_r}^{t_r}$ then the dual basis vector corresponding to $w_i$ is denoted as $T(r)_{w_i}$. Thus we have identified $(\tensor_{i=1}^s(U_{b_i}^{t_i})^{\tensor m_i})^*$ with $\tensor_{i=1}^s((U_{b_i}^{t_i})^*)^{\tensor m_i} $  by the above isomorphism.  More explicitly, for an index set $L=(l_1,\ldots, l_{\sum_i m_ib_i}),$ $ U=(u_1,\ldots, u_{\sum_i m_i t_i})$  the above isomorphism takes $T_{l_1,\ldots,l_{b_1},u_1\ldots, u_{t_1},l_{b_1+1},\ldots ,l_{2b_1},\ldots}$ 
to $(-1)^{p(\underline{w},\underline{w})}T(1)_{l_1\ldots l_{b_1}}^{u_1\ldots u_{t_1}}\tensor T(1)_{l_{b_1+1}\ldots l_{2b_1}}^{u_{t_1+1}\ldots u_{2t_1}}\tensor\cdots\in \tensor_{i=1}^s((U_{b_i}^{t_i})^*)^{\tensor m_i} $.  We then apply the permutation $\mu$ as defined earlier, which induces an isomorphism from $\tensor_{i=1}^s(U_{b_i}^{t_i})^{\tensor m_i}\to U^{\tensor (\sum_i m_ib_i)}\tensor (U^*)^{\tensor(\sum_i m_it_i)} $ and hence also on their duals. For the index set $L=(l_1,\ldots, l_{\sum_i m_ib_i}),$ $ U=(u_1,\ldots, u_{\sum_i m_i t_i})$  this isomorphism \begin{equation}\label{eq:2}
     \left(U^{\tensor (\sum_i m_ib_i)}\tensor (U^*)^{\tensor(\sum_i m_it_i)}\right)^*\to (\tensor_{i=1}^s(U_{b_i}^{t_i})^{\tensor m_i})^*
\end{equation} takes the element $T_{L,U}$
to  $\gamma(\mu^{-1}.(L,U), \mu^{-1})T_{l_1,\ldots,l_{b_1},u_1\ldots, u_{t_1},l_{b_1+1},\ldots ,l_{2b_1}},\ldots\in(\tensor_{i=1}^s(U_{b_i}^{t_i})^{\tensor m_i})^*$. A standard argument as outlined at the end of \S3 gives us that $U^{\tensor (\sum_i m_ib_i)}\tensor (U^*)^{\tensor(\sum_i m_it_i)}$ has invariants if and only if $\sum_i m_ib_i=\sum_i m_it_i$. From Theorem~1.1 we get a spanning set of $\Endo(U^{\tensor N})^{\gl(U)}$. So, in the case when $\sum_i m_ib_i=\sum_i m_it_i=N$, going by the isomorphism in (7) of \S \ref{iso}, we get invariants in $\left(U^{\tensor N}\tensor (U^*)^{\tensor N}\right)^*$ given by $\Theta(\sigma)$, where $\sigma \in S_N$ is regarded as an element of $\Endo_{\Lambda}(U^{\otimes N})$. As noted in \S \ref{iso} above, $\Theta(\sigma)(\underline{v}\tensor \underline{f})=ev[\nu.\tau.(\sigma(\underline{v})\tensor \underline{f}^*)]$. It can be easily seen that $\Theta(\sigma)$ can be expressed in terms of the basis of $(U^{\tensor N}\tensor (U^*)^{\tensor N})^*$ as \begin{equation}\label{eq:3}
    \sum_{L=(l_1,\ldots, l_N)}\gamma(L,\sigma^{-1})\gamma(\sigma L,\sigma L,\tau^{-1}) \gamma(\tau.(\sigma L,\sigma L),\nu^{-1})T_{L,\sigma L}.
\end{equation}Now going through the isomorphisms given in (\ref{eq:1}), (\ref{eq:2}), it can be seen that this element corresponds to 
the graded picture invariants defined above.
\end{proof}

\begin{remark}
In the above proof the invariants $\Theta(\sigma)$ are the images  of $\theta_{\sigma}$ in \S \ref{s.fft.gln} under the isomorphism from $(U^{\tensor N}\tensor (U^*)^{\tensor N})$ to its dual $(U^{\tensor N}\tensor (U^*)^{\tensor N})^*$. 
    
\end{remark}

\begin{remark}\label{r:linear}
      The linear map on $\tensor_{i=1}^s(U_{b_i}^{t_i})^{\tensor m_i}$ obtained by applying the isomorphism in \emph{(\ref{eq:2})} to $\Theta(\sigma)$ in the above proposition is obtained by applying a sequence of permutations, namely $\nu.\tau.\hat{\sigma}.\mu$, followed by the evaluation map. In the case that $W=\dsum_{i=1}^sU_1^1$, we note that $\mu^{-1}$ is just the permutation $\tau$. 
\end{remark}
\subsection{Restitution map and realizing linear maps as multilinear polynomials on $W_0$}
Following \cite[\S3]{LZ}, we consider the $\Lambda$-module of functions from $W\to \Lambda$ denoted by $\mathcal{F}(W,\Lambda)$. Let $F^r:T^r(W^*)\to \mathcal{F}(W,\Lambda)$ be given by $F^r({\boldsymbol{\alpha}})(v)=ev({\boldsymbol{\alpha}}\tensor v\tensor\cdots\tensor v)$ where $ev:T^{r}(W^*)\tensor T^r(W)\to \Lambda$ is same as in (3) above. It is noted in \cite[Lemma~3.11]{LZ} that $F^r$ restricts to map $F^r:S^r(W^*)\to \mathcal{F}(W_0)$; further, it is shown in \cite[Prop~3.14]{LZ} that this restriction is an isomorphism from $S^r(W^*)\to \mathcal{P}^r(W_0)$ where $\mathcal{P}^r(W_0)$ is the image of $F^r$ in $\mathcal{F}(W_0,\Lambda)$. The space $\mathcal{P}^r(W_0)$ is called the space of ``polynomials of degree $r$ on $W_0$".  For $f\in S^r(W^*)$ and $w\in W_0$, we note that $F^r(f)(w)=ev(\boldsymbol{f}\tensor w\tensor w\tensor w\tensor\ldots r\textit{-times})$ where $\boldsymbol{f}\in T^r(W^*)$ such that $\varpi_r(\boldsymbol{f})=f$ where $\varpi_r: T^r(W^*)\to S^r(W^*)$ is the natural quotient map.

Let $W=\dsum_{i=1}^sU_{b_i}^{t_i}$. Since $W^*\iso \dsum_{i=1}^s(U_{b_i}^{t_i})^*$, the set of all linear maps $T(i)_{l_1\ldots l_{b_i}}^{u_1\ldots, u_{t_1}}\in (U_{b_i}^{t_i})^*$, for each $i=1,\ldots, s$,  $1\leq l_j\leq m+n$ and $1\leq u_j\leq m+n$ gives a $\Lambda$-basis of $W^*$, which shall also be denoted by the same notation; regarded as a linear map on $W^*$, $T(i)_{l_1\ldots l_{b_i}}^{u_1\ldots, u_{t_1}}(u_1,\ldots,u_s):=T(i)_{l_1\ldots l_{b_i}}^{u_1\ldots, u_{t_1}}(u_1)$. Setting the multidegree of $T(i)_{l_1\ldots l_{b_i}}^{u_1\ldots u_{t_1}}\in W^*$ to be $(0,\ldots, 0,1,0\ldots, 0)$ with $1$ in the $i^{\textup{th}}$-position and $0$ elsewhere, denote by $S^{(m_1,\ldots,m_s)}(W^*)$ the  $(m_1,\ldots, m_s)$-multidegree summand of $S^r(W^*)$. Then monomials $\prod_{i=1}^sT(i)_{l_1\ldots l_{b_i}}^{u_1\ldots u_{t_i}}T(i)_{l_{b_i+1}\ldots l_{2b_i}}^{u_{t_i+1}\ldots u_{2t_i}}\cdots T(i)_{l_{(m_i-1)b_i+1}\ldots l_{m_ib_i}}^{u_{(m_i-1)t_i+1}\ldots u_{m_it_i}}$ where the $l_j$'s and $u_j$'s vary in $\{1,\ldots, m+n\}$ forms a basis of $S^{(m_1,\ldots,m_s)}(W^*)$. It is obvious that there is a natural isomorphism between $\tensor_{i=1}^sS^{m_i}((U_{b_i}^{t_i})^*)$ and $S^{(m_1,\ldots,m_s)}(W^*)$. Also, note that $\tensor_{i=1}^s((U_{b_i}^{t_i})^*)^{\tensor m_i}\hookrightarrow T^r(W^*)$ where $r=\sum_{i=1}^sm_i$.  
\begin{lemma}\label{l:tensorsymm}
    With notations as above, the following diagram commutes\[\xymatrix{\tensor_{i=1}^s((U_{b_i}^{t_i})^*)^{\tensor m_i}\ar@{->>}[r]^{\varpi_{m_1}\tensor\ldots\tensor \varpi_{m_s}}\ar@{^{(}->}[dd]&\tensor_{i=1}^sS^{m_i}((U_{b_i}^{t_i})^*)\ar[d]^{\cong}\\&S^{(m_1,\ldots, m_s)}(\dsum_{i=1}^s(U_{b_i}^{t_i})^*)\ar@{^{(}->}[d]\\T^r(W^*)\ar@{->>}[r]^{\varpi_r}&S^r(W^*)
}\]
\end{lemma}
\begin{proof}
    It is an easy checking to observe that the basis elements \[T(1)_{l_1\ldots l_{b_1}}^{u_1\ldots u_{t_1}}\tensor \ldots \tensor T(i)_{l_{\sum_{j<i}m_jb_j+(kb_i)+1}\ldots l_{\sum_{j<i}m_jb_j+(k+1)b_i)}}^{u_{\sum_{j<i}m_jt_j+(kt_i)+1}\ldots u_{\sum_{j<i}m_jt_j+(k+1)t_i)}}\tensor \ldots\tensor T(s)_{l_{\sum_{j<s}m_jb_j+((m_s-1)b_s)+1}\ldots l_{\sum_{j<i}m_jb_j+m_sb_s)}}^{u_{\sum_{j<s}m_jt_j+((m_s-1)t_s)+1}\ldots u_{\sum_{j<i}m_jt_j+m_st_s)}}\] of $\tensor_{i=1}^s((U_{b_i}^{t_i})^*)^{\tensor m_i}$
    goes to the same monomial via both $\varpi_r$ and $\varpi_{m_1}\tensor\ldots\tensor \varpi_{m_s}$, as required.
\end{proof}
\begin{lemma}\label{l:restitution}
With notations as above, for each $s$-tuple of non-negative integers $(m_1,\ldots, m_s)$ 
    the map $F^r$ restricts to an isomorphism from $S^{(m_1,\ldots,m_s)}(W^*)\to \mathcal{P}^{(m_1,\ldots,m_s)}(W_0)$ where $\mathcal{P}^{(m_1,\ldots,m_s)}(W_0)$ denotes the image of $S^{(m_1,\ldots,m_s)}(W^*)$ in $\mathcal{F}(W_0,\Lambda)$. Further under the natural identification of $S^{(m_1,\ldots,m_s)}(W^*)$ with $\tensor_{i=1}^sS^{m_i}((U_{b_i}^{t_i})^*)$,  the map $F^{(m_1,\ldots,m_s)}:=F^r\circ(\varpi_{m_1}\tensor \ldots\tensor \varpi_{m_s}):T^{m_1}((U_{b_1}^{t_1})^*)\tensor\ldots\tensor T^{m_s}((U_{b_s}^{t_s})^*)\to \mathcal{P}^{(m_1,\ldots,m_s)}(W_0)$ is given by $F^{(m_1,\ldots,m_s)}(\boldsymbol{f_1}\tensor\ldots\tensor \boldsymbol{f_s})(u_1,\ldots,u_s)=ev(\boldsymbol{f_1}\tensor (u_1\tensor u_1\tensor \ldots m_1{\textit{-times}})) \cdots ev(\boldsymbol{f_s}\tensor (u_s\tensor u_s\tensor \ldots m_s{\textit{-times}}))$. In other words, the following diagram commutes 
    \[\xymatrix{&\tensor_{i=1}^sS^{m_i}((U_{b_i}^{t_i})^*)\ar[d]^{\cong}\ar[r]^{F^{m_1}\tensor\ldots\tensor F^{m_s}  }&\tensor_{i=1}^s\mathcal{F}((U_{b_i}^{t_i})_0,\Lambda)\ar[dd]\\\tensor_{i=1}^s((U_{b_i}^{t_i})^*)^{\tensor m_i}\ar@{->>}[ru]^{\varpi_{m_1}\tensor\ldots\tensor \varpi_{m_s}}\ar@{->}[rd]&S^{(m_1,\ldots, m_s)}(\dsum_{i=1}^s(U_{b_i}^{t_i})^*)\ar@{^{(}->}[d]\\ &S^r(W^*)\ar[r]^{F^r}&\mathcal{F}(W_0,\Lambda)
}\]
\end{lemma}
\begin{proof}The first part is a consequence of the fact noted earlier that $F^r:S^r(W^*)\to \mathcal{F}(W_0)$ is injective. For the second part, note that under the natural isomorphism between $\tensor_{i=1}^sS^{m_i}((U_{b_i}^{t_i})^*)$ and $S^{(m_1,\ldots,m_s)}(W^*)$, the monomial $f_1\cdots f_s \in S^{(m_1,\ldots,m_s)}(W^*)$ can be identified with the element $f_1\tensor\cdots \tensor f_s \in \tensor_{i=1}^sS^{m_i}((U_{b_i}^{t_i})^*)$ where $f_i\in S^{m_i}((U_{b_i}^{t_i})^*)$. Then the commutativity of the diagram in Lemma~\ref{l:tensorsymm} implies that $ F^r(f_1\cdots f_s)(\boldsymbol{u})=F^r({f_1}\tensor \cdots \tensor {f_s})(\boldsymbol{u})=ev(\boldsymbol{f_1}\tensor \cdots \tensor \boldsymbol{f_s}\tensor \boldsymbol{u}\tensor\boldsymbol{u}\tensor \ldots  \textup{($\sum_im_i$-times)} )=ev(\boldsymbol{f_1}\tensor \boldsymbol{u}\tensor\boldsymbol{u}\tensor\ldots \textup{($m_1$-times)}\ldots\tensor\cdots \tensor \boldsymbol{f_s}\tensor\boldsymbol{u}\tensor \boldsymbol{u}\tensor\ldots \textup{($m_s$ -times)}\ldots)$
where $\boldsymbol{u}=(u_1,\ldots,u_s)\in W_0$ and ${\boldsymbol{f_i}}\in T^{m_i}((U_{b_i}^{t_i})^*)$ such that $\varpi_{m_i}(\boldsymbol{f_i})=f_i$. Note that, since $\boldsymbol{u}$ is a degree $0$ element there is no sign change in the above equality. The right side of the above equality evaluates to $\mathcal{F}^{m_1}({\boldsymbol{f_1}})(u_1)\cdots \mathcal{F}^{m_1}({\boldsymbol{f_s}})(u_s)$. 
\end{proof}

\begin{remark}\label{r:restitution} We note from the commutativity of the above diagram and with the notations as in the above lemma, $F^r(\boldsymbol{T})(\boldsymbol{u})=ev(\boldsymbol{T}\tensor u_1\tensor u_1\tensor \cdots \textup{($m_1$-times)} \tensor\cdots\tensor u_s\tensor u_s\tensor\cdots  \textup{($m_s$-times)}$ for $\boldsymbol{T}\in \tensor_{i=1}^sT^{m_i}((U_{b_i}^{t_i})^*)$ and $\boldsymbol{u}=(u_1,\ldots,u_s)\in W$ 
\end{remark}

The image of $S^{(m_1,\ldots,m_s)}(W^*)$ in $\mathcal{F}(W_0,\Lambda)$ is the space of ``multihomogeneous polynomials on $W_0$", denoted as $\mathcal{P}^{(m_1,\ldots,m_s)}(W_0)$. And the map $F^r$ corresponds to the restitution map  $F_1^{m_1}\tensor \cdots\tensor F_s^{m_s}:\tensor_{i=1}^sS^{m_i}((U_{b_i}^{t_i})^*)\to \mathcal{P}^{(m_1,\ldots,m_s)}(W_0)$. When $m_i=1$ for all $i1,\ldots, s$, this is the space of multilinear polynomials on $W_0$. 

\subsubsection{Polynomial invariants and `pictures'.} By a polynomial invariant on $W_0$, we shall mean an element of $\mathcal{F}(W_0,\Lambda)$ which is in the image of $S(W^*)^{\gl(U)}$. By Theorem~\ref{t:picture}, this means that the subalgebra of invariant polynomials on $W_0$ is spanned by $F^r(\varphi_\sigma)$ where $\varphi_\sigma$ are the picture invariants as defined above. Let $T_\sigma$ be the linear map described in Remark~\ref{r:linear}. Then we have,
\begin{proposition}\label{l:restitution2}
    The graded picture invariants $\varphi_\sigma\in S^{(m_1,\ldots,m_s)}(W^*)$ maps under restitution $F^r$ to the map $T_\sigma$ given by $T_\sigma(u_1\tensor u_1\tensor \cdots \textup{($m_1$-times)} \tensor\cdots\tensor u_s\tensor u_s\tensor\cdots  \textup{($m_s$-times)})$ where $\boldsymbol{u}=(u_1\ldots,u_s)\in W_0$.
\end{proposition}
\begin{proof}
    The isomorphism  $\iota:T^r(W^*)\to (T^r(W))^*$ used in Theorem~\ref{t:picture} is obtained from the non-degenerate pairing $T^r(W^*)\tensor T^r(W)\to\Lambda$, as in (4) above. Therefore, for any $\phi\in T^r(W^*)$ we get a linear map $\iota(\phi)$ on $T^r(W)$ and the evaluation $\iota(\phi)(w_1\tensor\ldots\tensor w_r)$ is given by $ev(\phi\tensor w_1\tensor\ldots\tensor w_r)$. Hence we have the following commuting diagram:
 \[\xymatrix{\tensor_{i=1}^s((U_{b_i}^{t_i})^*)^{\tensor m_i}\ar@{^{(}->}[r]&T^r(W^*)\ar[r]&S^r(W^*)\ar[r]^{F^r}&\mathcal{F}(W_0,\Lambda)\\ \left(\tensor_{i=1}^s((U_{b_i}^{t_i})^{\tensor m_i}\right)^*\ar@{^{(}->}[r]\ar[u]&(T^r(W))^*\ar[u]^{\cong}\ar@{-->}[rru]
}\]where the inclusion $\tensor_{i=1}^s((U_{b_i}^{t_i})^*)^{\tensor m_i}\hookrightarrow T^r(W^*)$ is as earlier while the map $\left(\tensor_{i=1}^s((U_{b_i}^{t_i})^{\tensor m_i}\right)^*\hookrightarrow (T^r(W))^*$ is induced by the projection  $\textup{pr}: T^r(W)\twoheadrightarrow \left(\tensor_{i=1}^s((U_{b_i}^{t_i})^{\tensor m_i}\right)$;  for $\psi\in \left(\tensor_{i=1}^s((U_{b_i}^{t_i})^{\tensor m_i}\right)^* $ and $w_1\tensor \ldots \tensor w_r\in T^r(W)$, $\psi(w_1\tensor \ldots \tensor w_r):=\psi(\textup{pr}(w_1\tensor \ldots \tensor w_r))$. By the construction of $\varphi_\sigma$ in Theorem~\ref{t:picture}, we know that $T_\sigma\in \left(\tensor_{i=1}^s((U_{b_i}^{t_i})^{\tensor m_i}\right)^*$ maps to $\varphi_\sigma\in S^r(W^*)$. Let $\boldsymbol{\varphi_\sigma}\in T^r(W^*)$ such that $\iota(\boldsymbol{\varphi_\sigma})=T_\sigma$ and $\varpi_r(\boldsymbol{\varphi_\sigma})=\varphi_\sigma$. By Remark~\ref{r:restitution} we have $F^r(\varphi_\sigma)(\boldsymbol{u})=ev(\boldsymbol{\varphi_\sigma}\tensor u_1\tensor u_1\tensor \cdots \textup{($m_1$-times)} \tensor\cdots\tensor u_s\tensor u_s\tensor\cdots  \textup{($m_s$-times)})$ where $\boldsymbol{u}=(u_1\ldots,u_s)\in W_0$. Under the isomorphism $\iota$, as seen above, the latter evaluates to $T_\sigma( u_1\tensor u_1\tensor \cdots \textup{($m_1$-times)} \tensor\cdots\tensor u_s\tensor u_s\tensor\cdots  \textup{($m_s$-times)})$,  as required.
\end{proof}
By the above proposition, the invariant polynomials on $W_0$ are given by the evaluations of the linear maps $T_\sigma:=ev\circ \eta(\nu\tau\hat{\sigma}\mu)$ where $\eta:=\eta_{2N}$ is as in \S\ref{action}. In the spirit of \cite{dks}, we associate certain pictures to these invariants. To each variable $T(i)_{l_1\ldots l_{b_i}}^{u_1\ldots u_{t_i}} \in W^*$ we associate a picture\\
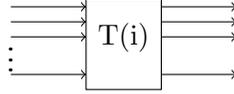
\begin{center}
    \begin{tikzpicture}
\draw node{$\vdots$} [->](0,-0.3) -- (1,-0.3);
\draw [->](0,0.4) -- (1,0.4);
\draw [->](0,0.6) -- (1,0.6);
\draw [->](0,0.2) -- (1,0.2);
\draw (1,-0.5) rectangle +(1,1.2);
\draw (1.5,0.2) node {T(i)};
\draw [->](2,-0.3) -- (3,-0.3);

\draw node{$\vdots$} [->](2,0.6) -- (3,0.6);
\draw [->](2,0.4) -- (3,0.4);
\draw [->](2,0.2) -- (3,0.2);
\end{tikzpicture}\\
\end{center}
\noindent
with $t_i$ in-arrows and $b_i$ out-arrows. Given an $s$-tuple $(m_1,\ldots, m_s)$ of non-negative integers such that $\sum_{i=1}^sm_it_i=\sum_{i=1}^sm_ib_i=N$, we take for each i, $m_i$ many copies of the pictures associated to $T(i)_{l_1\ldots l_{b_i}}^{u_1\ldots u_{t_i}}$ and arrange them in order from $i=1$ to $i=s$ and then label all the in-arrows and all the out-arrows from $1$ to $N$ in the order that they appear. Now given a $\sigma\in \Sym_N$ we associate a closed picture obtained by joining the in-arrow labelled $i$ with the out-arrow labelled $\sigma^{-1}(i)$. For example, let $W= V_2^1\dsum V_1^0\dsum V_0^1$ and $m_1=1; m_2=2;m_3=1$. Then corresponding to $\sigma=(1~2~3)$ we get the closed picture below:\\

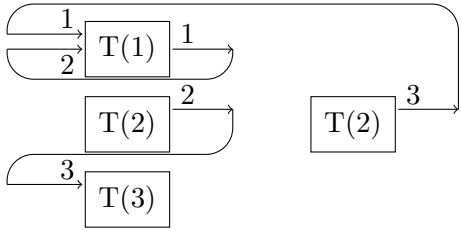
\begin{tikzpicture}
\draw [->](0,0)--(1,0);
\draw [->](0,1.8)--(1,1.8);
\draw [->](2.2,1.8)--(3,1.8);
\draw [->](2.2,1)--(3,1);
\draw [->](5.2,1)--(6,1);
\draw [->](0,2)--(1,2);
\draw (0.8,0.2) node(i3) {3} (0.8,1.6) node(i2) {2} (0.8, 2.2) node(i1) {1} (2.4,2) node(o1) {1} (2.4,1.2) node(o2) {2} (5.4,1.2) node(o3) {3};
\draw[rounded corners=10pt] (0,0)--(0,0.4)--(3,0.4)--(3,1);
\draw[rounded corners=10pt] (3,1.8)--(3,1.4)--(0,1.4)--(0,1.8);
\draw[rounded corners=10pt] (6,1)--(6,2.4)--(0,2.4)--(0,2);
    \draw (1.6,-0.2) node[inner sep=5pt, draw]{T(3)}
    (1.6,0.8) node[inner sep=5pt, draw]{T(2)}
    (4.6,0.8) node[inner sep=5pt, draw]{T(2)}
    (1.6,1.8) node[inner sep=5pt, draw]{T(1)};
\end{tikzpicture}\\
\noindent
This picture corresponds to the graded picture invariant $\varphi_\sigma$ as defined above.
\subsubsection{Graded picture invariants in terms of supertraces}
We now restrict to the case when $W=\dsum_s (U_1^1)$. Under the identification of (1) above, $U_1^1\iso End_\Lambda(U)$. We can then consider the supertrace function of $U_1^1$ given by $str(v\tensor\alpha)=(-1)^{|v||\alpha|}\alpha(v)$ and $(v\tensor \alpha). (w\tensor \beta)=v\alpha(w)\tensor\beta$. With this notation, one may define the trace monomial $tr_\sigma$ for a permutation $\sigma\in S_N$ as $tr_\sigma(v_{1}\tensor \phi_1\tensor\cdots \tensor v_{N}\tensor \phi_{N})=str(v_{{i_1}}\tensor \phi_{{i_1}}.v_{{i_2}}\tensor \phi_{{i_2}}.\cdots .v_{{i_r}}\tensor \phi_{{i_r}})str(v_{{j_1}}\tensor \phi_{{j_1}}.v_{{j_2}}\tensor \phi_{{j_2}}.\cdots \tensor v_{{j_t}}\tensor \phi_{{j_t}})$
where $\sigma^{-1}=(i_1 ~i_2~ \ldots i_r)(j_1~j_2\cdots j_t)\cdots$. This definition is dependent on the permutation $\sigma$ and its cycle decomposition as well.  We show that 
\begin{lemma}\label{l:trace} (see \cite[Lemma~2.2]{Berele})
For a $\sigma\in \Sym_N$ such that $\sigma^{-1}=(i_1 ~i_2~ \ldots i_r)(j_1~j_2\cdots j_t)\cdots\in \Sym_N$, the $\gl(U)$-invariant map $ev\circ \eta(\nu.\tau.\hat{\sigma}.\tau^{-1})$ (as in Remark~\emph{\ref{r:linear}}) corresponds to the trace monomials $tr_\sigma$ upto a sign. Further, both the maps agree when restricted to the degree $0$ part, $((U_1^1)_0)^{\tensor N}$.
\end{lemma}
\begin{proof}
For the first part, by $\Lambda$-linearity it suffices to prove the above for the basis elements $(e_{l_1}\tensor e^*_{u_1}\tensor \cdots e_{l_N}\tensor e^*_{u_N})$. 
The effect of applying $\nu.\tau.\hat{\sigma}.\tau^{-1}$ to $e_{l_1}\tensor e^*_{u_1}\tensor \cdots e_{l_N}\tensor e^*_{u_N}$ from the left is to get $\pm (e^*_{u_1}\tensor e_{l_{\sigma^{-1}(1)}}\tensor \cdots e^*_{u_N}\tensor e_{u_{\sigma^{-1}(N)}})$. Now applying the evaluation map $ev$ on this vector gives, $\pm e^*_{u_1}(e_{l_{\sigma^{-1}(1)}})\cdots e^*_{u_{r}}(v_{l_{\sigma^{-1}(r)}})e^*_{u_{r+1}}(e_{l_{\sigma^{-1}(r+1)}})\cdots$.
After a rearrangement, the latter is equal to $e^*_{u_{i_1}}(e_{l_{i_2}})\cdots e^*_{u_{i_r}}(e_{l_{i_1}})e^*_{u_{j_{1}}}(e_{l_{j_{2}}})\cdots$, upto a sign. This in turn is the trace monomial $tr_\sigma(e_{l_1}\tensor e^*_{u_1}\tensor \cdots e_{l_N}\tensor e^*_{u_N})$.

For the second part, we try to determine the sign change explicitly. First, we consider the special case of $\sigma^{-1}=(1~2~\cdots r)(r+1 \cdots t)\cdots$. Let $v_i\tensor\phi_i\in V\tensor V^*$. Let $v_{i}$ have parity $a_i$ then $\phi_{i}$ have parity $a'_i$. The effect of applying $\nu.\tau.\hat{\sigma}.\tau^{-1}$ to $v_{1}\tensor \phi_{1}\tensor\cdots \tensor v_{N}\tensor \phi_{N}$ on the left is to get $\pm  (\phi_{1}\tensor v_{\sigma^{-1}(1)}\tensor \phi_N\tensor v_{\sigma^{-1}(N)})$. We consider the special case of $\sigma^{-1}=(1~2~\cdots r)(r+1 \cdots t)\cdots$. In this case, acting $\nu.\tau.\hat{\sigma}.\tau^{-1}$ from the left on $(v_{1}\tensor \phi_{1}\tensor \cdots v_{N}\tensor \phi_{N})$ we get $\pm( \phi_{{1}}\tensor v_{{2}}\tensor \phi_{{2}}\tensor v_{3}\cdots \tensor \phi_{{r}}\tensor v_{{1}})\tensor (\phi_{r+1}\tensor v_{{r+2}}\tensor\cdots\tensor \phi_{{t}} \tensor v_{{r+1}})\tensor\cdots$.  Now applying the evaluation map $ev$ on this vector gives, $\pm \phi_{1}(v_{2})\cdots \phi_{r}(v_{1})\phi_{r+1}(v_{{r+2}})\cdots$, leaving the sign unchanged. The total sign change is then given by $a_1(a'_1+a_2 +a'_2+\cdots +a_{r}+a'_{r})+a_{r+1}(a'_{r+1}+a_{r+2}+\cdots+a'_t)+\cdots$. If $v_i\tensor \phi_i$ are of degree $0$ then $a_i=a'_i$ for each $i$. So, for these values we see that the sign is $(-1)^{a_1+a_{r+1}+\cdots}$. This agrees with the value of the trace monomial $tr_\sigma$ evaluated on $(v_{l_1}\tensor \phi_{u_1}\tensor \cdots v_{l_N}\tensor \phi_{u_N})$. Thus, in the case that $\sigma^{-1} =(1~2~\cdots r)(r+1 \cdots t)\cdots$ the linear maps $ev\circ \eta(\nu.\tau.\hat{\sigma}.\tau^{-1})$ and $tr_\sigma$ agree on  $\left((U_1^1)_0\right)^{\tensor N}$. The proof of the general case is same as in \cite[Lemma2.2]{Berele}, which we recall for completeness: 
Let $\sigma^{-1}=(i_1 ~i_2~ \ldots i_r)(i_{r+1}~i_{r+2}\cdots i_t)\cdots\in \Sym_N$. Let $\pi\in\Sym_N$ such that $\pi(j)=i_j$ then $\pi\sigma^{-1}\pi^{-1}=(1~2~\cdots r)(r+1\cdots t)\cdots$. We have, $tr_{\pi\sigma\pi^{-1}}(v_{\pi(1)}\tensor \phi_{\pi(1)}\tensor\cdots\tensor v_{\pi(N)}\tensor\phi_{\pi(N)}):=str(v_{\pi(1)}\tensor \phi_{\pi(1)}\tensor \cdots\tensor v_{\pi(r)}\tensor\phi_{\pi(r)})str(v_{\pi(r+1)}\tensor \phi_{\pi(r+1)}\tensor \cdots\tensor v_{\pi(t)}\tensor\phi_{\pi(t)})\cdots=tr_{\sigma}(v_{1}\tensor \phi_{1}\tensor \cdots v_{N}\tensor \phi_{N})$. By the special case proved earlier, we have $tr_{\pi\sigma\pi^{-1}}(v_{\pi(1)}\tensor \phi_{\pi(1)}\tensor\cdots\tensor v_{\pi(N)}\tensor\phi_{\pi(N)})=(ev\circ \eta(\nu\tau(\hat{\pi\sigma\pi^{-1}})\tau^{-1}))(v_{\pi(1)}\tensor \phi_{\pi(1)}\tensor\cdots\tensor v_{\pi(N)}\tensor\phi_{\pi(N)})=ev( (v_{\pi(1)}\tensor \phi_{\pi(1)}\tensor\cdots\tensor v_{\pi(N)}\tensor\phi_{\pi(N)})\tau.(\hat{\pi\sigma\pi^{-1}})^{-1}.\tau^{-1}.\nu)=ev( (v_{1}\tensor \phi_{1}\tensor\cdots\tensor v_{N}\tensor\phi_{N}).\pi^*.\tau(\hat{\pi\sigma\pi^{-1}})^{-1}.\tau^{-1}.\nu)$ where $\pi^*=\tau (\pi^{-1},\pi^{-1})\tau^{-1}$. We have \begin{align*}\tau (\pi^{-1},\pi^{-1})\tau^{-1}\tau(\Hat{\pi\sigma\pi^{-1}})^{-1}\tau^{-1}.\nu&=\tau (\pi^{-1},\pi^{-1})(\pi\sigma^{-1}\pi^{-1},1)\tau^{-1}.\nu\\ 
&=\tau(\sigma ^{-1}\pi^{-1},\pi^{-1})\tau^{-1}\nu\\ 
&=\tau\hat{\sigma}^{-1}(\pi^{-1},\pi^{-1})\tau^{-1}\nu\\ 
&=\tau\hat{\sigma}^{-1}\tau^{-1}\tau(\pi^{-1},\pi^{-1})\tau^{-1}\nu\\
&=\tau\hat{\sigma}^{-1}\tau^{-1}\pi^*\nu\\ &=\tau\hat{\sigma}^{-1}\tau^{-1}\nu\pi^*. 
\end{align*} The last equality comes from noting that $\pi^*\nu$ and $\nu\pi^*$ agree on vectors of the form $v_{1}\tensor \phi_{1}\tensor\cdots \tensor v_{N}\tensor \phi_{N}$. Also, $ev([v_{1}\tensor \phi_{1}\tensor\cdots \tensor v_{N}\tensor \phi_{N}]\pi^*)=ev(v_{1}\tensor \phi_{1}\tensor\cdots \tensor v_{N}\tensor \phi_{N})$ on $(V\tensor V^*)^{\otimes N}$. Putting all the above observations together, we have $tr_\sigma(v_{1}\tensor \phi_{1}\tensor\cdots \tensor v_{N}\tensor \phi_{N})=(ev\circ\eta(\tau\hat{\sigma}\tau^{-1}\nu))(v_{1}\tensor \phi_{1}\tensor\cdots \tensor v_{N}\tensor \phi_{N}))$ for all $\sigma\in \Sym_N$, proving the second part.
\end{proof}
It is shown in \cite[Lemma~6.6]{Berele2} that in the special case of $W=(U_1^1)^d$ and when $m=n$, that $\mathcal{P}^{(1,\ldots,1)}(W_0)$ is actually the space of multilinear polynomials in the co-ordinates of $M_{2n}(\Lambda)_{0}=M_{n,n}$. The following corollary may be viewed as a generalization of this result.
\begin{corollary}\label{c:traces}
    The picture invariants $\varphi_\sigma\in S^{(m_1,\ldots,m_s)}(W^*)$ maps under restitution $F^r$ to the trace monomial $T_\sigma$ given by $T_\sigma(\boldsymbol{u})=tr_\sigma(u_1\tensor u_1\tensor \cdots \textup{($m_1$-times)} \tensor\cdots\tensor u_s\tensor u_s\tensor\cdots  \textup{($m_s$-times)})$ where $\boldsymbol{u}=(u_1\ldots,u_s)\in W_0$.
\end{corollary}
\begin{proof}
The proof follows from Proposition~\ref{l:restitution2} and the above lemma.
\end{proof}
 The invariants in $P(W_0)$ for induced action of $\gl(U)$ such that the isomorphism is $\gl(U)$-equivariant are called the invariant polynomials on $W_0$. The following is Theorem~6.7 of \cite{Berele2},
\newcommand{\str}{\textup{str}}
\begin{theorem}\label{t:trace}
    The invariant polynomials for the simultaneous action of $\gl(U)$ on $\oplus_{i=1}^s(U_1^1)_0$ is spanned by the trace monomials $\tr_\sigma$ given by \[\tr_\sigma(A_1,\ldots, A_s):= \str(A_{f(i_1)}\cdots A_{f(i_r)})\str(A_{f(i_{r+1})}\cdots A_{f(i_t)})\cdots\] where $A_1,\ldots, A_s\in\oplus_{i=1}^sU_1^1$, $\sigma=(i_1 \ldots i_r)(i_{r+1} \ldots i_t)\cdots\in \Sym_n$ and a map $f:\{1,\ldots, n\}\to \{1,\ldots, s\}$.
\end{theorem}
\begin{proof}
    The invariants in $\mathcal{P}(W_0)$ is the image of $S(W^*)^{\gl(U)}$ which in turn is spanned by $\varphi_\sigma$, by Theorem~\ref{t:picture}. By the above Corollary~\ref{c:traces} we then get the required result.
\end{proof}

\section{Schur-Weyl-Sergeev Duality for queer superalgebra}

\subsection{Queer superalgebra} Let $V=V_0 \oplus V_1$ be the $\mathbb Z_2$-graded superspace for which $V_0=\mathbb C^n$ and $V_1=\mathbb C^n$. 
Let $\{v_1,\cdots ,v_n\}$ and $v_{-1}, \cdots ,v_{-n}\}$ be bases for $V_0$ and $V_1$ respectively. 

Define a map $P:V \rightarrow V$ by $v_i \mapsto (-1)^{|v_i|}v_{-i}$ for all $i$. 

Recall that the superbracket in $End_{\mathbb C}(V)$ is given by $$[f,g]=fg-(-1)^{|f||g|}gf.$$ for homogeneous elements $f,g \in End_{\mathbb C}(V)$.

Let $U:=V \otimes_{\mathbb C} \Lambda$. The map $P \in End_{\mathbb C}(V)$ extends to a map in $End_{\Lambda}(U)$ which we still denote by $P$. We have $P^2=-I$. 

We define $\mathfrak q(U):=\{f \in End_{\Lambda}(U)_0: [f,P]=0\}$. It is closed under the superbracket and we call it the queer superalgebra.

With respect to the basis $B:=B_0 \cup B_1$, the operator $P$ can be written in matrix form as $$P=\begin{pmatrix}
0 & -I \\
I & 0
\end{pmatrix}$$ and the queer Lie superalgebra $\mathfrak q(U)$ (or $\mathfrak q(n)$) can be expressed in the matrix form $$\mathfrak q(n)=\{\begin{pmatrix}
A & B \\
B & A
\end{pmatrix}: A \in M_n(\Lambda_0), B \in M_n(\Lambda_1)\}.$$

We define the queer trace of an element of 
$\mathfrak q(n)$ to be $$qtr(\begin{pmatrix}
A & B \\
B & A
\end{pmatrix})=[tr(A+B)]_1,$$ where $[tr(A+B)]_1$ is the degree $1$ component of the sum of the diagonal entries of $A+B$. We can see that $qtr$ satisfy $qtr(MN)=qtr(NM)$ and $qtr(RMR^{-1})=qtr(M)$ for all $M,N,R \in \mathfrak q(n)$, where $R$ is invertible.


The queer superalgebra $\mathfrak q(U)$ is a subalgebra of $\mathfrak{gl}(U)$ and it acts on $U^{\otimes k}$ by  ordinary derivation (see \S \ref{action}).

The queer supergroup $Q(U)$ is the subsupergroup of $GL(U)$ which commute with $P$. The diagonal action of $GL(U)$ on $U^{\otimes k}$ restricts to an action of $Q(U)$ on $U^{\otimes k}$. 

\subsection{The Sergeev superalgebras}
Let $S_k$ be the symmetric group on $k$ letters. It is generated by the transpositions $s_1,\cdots ,s_{k-1}$. 

The Sergeev superalgebra $Ser_k$ is the associative superalgebra generated by $s_1,\cdots ,s_{k-1}$ and $c_1,\cdots ,c_{k-1},c_k$ with the following defining relations:

$s_i^2=1, s_is_{i+1}s_i=s_{i+1}s_is_{i+1}, s_is_j=s_js_i \,\,\,\, (|i-j|>1),$

$c_i^2=-1, c_ic_j=-c_jc_i, \,\, (i \neq j)$

$s_ic_is_i=c_{i+1}, s_ic_j=c_js_i, \,\, j \neq i,i+1$.

The generators $s_1,\cdots ,s_{k-1}$ are regarded as even and the subalgebra generated by them is isomorphic to the group algebra $\mathbb C[S_k]$ of $S_k$; the generators $c_1,\cdots ,c_{k-1},c_k$ are called odd and the $\mathbb C$-subalgebra generated by them is isomorphic to the Clifford superalgebra $\mathbf Cl_k$. 

The Sergeev superalgebra $Ser_k$ is isomorphic to the superalgebra $\mathbb C[S_k] \ltimes \mathbf Cl_k$, which is $\mathbb C[S_k] \otimes \mathbf Cl_k$ as a superspace with the multiplication given by

$$(\sigma \otimes c_{i_1}\cdots c_{i_l})(\tau \otimes c_{j_1}\cdots c_{j_m})=\sigma \tau \otimes c_{\tau^{-1}(i_1)}\cdots c_{\tau^{-1}(i_l)} c_{j_1}\cdots c_{j_m},$$ where $1 \leq i_s, j_t \leq k$.

Let $\Lambda Ser_k:=Ser_k \otimes_{\mathbb C} \Lambda$. There is a left action of the superalgebra $\Lambda Ser_k$ on $U^{\otimes k}$ which is given as follows (see \S 3.4.1 of \cite{CW}):  

The generators $s_j$ acts on $U^{\otimes k}$ by $$s_j.(u_1 \otimes \cdots \otimes u_k)=(-1)^{|u_j||u_{j+1}|}u_1 \otimes \cdots \otimes u_{j-1} \otimes u_{j+1} \otimes u_j \otimes u_{j+2} \otimes \cdots \otimes u_k.$$

The generators $c_j$ acts on $U^{\otimes k}$ by $$c_j.(u_1 \otimes \cdots \otimes u_k)=(-1)^{|u_1|+\cdots +|u_{j-1}|}u_1 \otimes \cdots \otimes u_{j-1} \otimes P(u_{j}) \otimes u_{j+1} \otimes u_{j+2} \otimes \cdots \otimes u_k.$$

The actions of $s_i$ and $c_j$ give rise to an action of the superalgebra $\Lambda Ser_k$ on $U^{\otimes k}$. So by $\Lambda$-linearity, we get an algebra homomorphism $$\Phi_k: \Lambda Ser_k \rightarrow \Endo_{\Lambda} (U^{\otimes k}).$$ 


Let $\Endo_{Q(U)} (U^{\otimes k})$ be the centralizer algebra of the $Q(U)$-action on $U^{\otimes k}$. That is $$\Endo_{Q(U)}(U^{\otimes k}):=\{f \in \Endo_{\Lambda} (U^{\otimes k}): gf=fg, \,\, \text{for all} \,\, g \in Q(V)\}.$$


Sergeev \cite{Ser3} proved the double centralizer theorem along the lines of the Schur–Weyl duality for the queer superalgebra and it was extended to the Queer supergroups by Berele (see Theorem 4.2 of \cite{Berele2}). 






\begin{theorem}\label{t:sergeev}
Let $\rho: Q(U) \rightarrow \Endo_{\Lambda}(U^{\otimes d})$ and $\Phi_k: \Lambda Ser_k \rightarrow \Endo_{\Lambda}(U^{\otimes d})$ be the maps induced from the actions defined above. Let $A$ be the image of $\Phi_k$ and let be $B$ the $\Lambda$-subalgebra of $\Endo_{\Lambda}(U^{\otimes d})$ spanned by the image of $\rho$. Then $A$ and $B$ are centralizers of each other.  
\end{theorem}

\begin{remark} It may be noted that in the statement of the above theorem we consider the left action of the Sergeev algebra whereas Berele considers the action from the right and uses $\Lambda Ser_k^{op}$ instead.
\end{remark}

\vspace{0.2cm}

\subsection{}\label{q.fact} In this subsection we list out some facts that are needed to prove the main results in \S6. We give an outline of the proofs here, for more details refer to \cite{Berele2}.

1. There is a $GL(U)$ equivariant isomorphism $\Endo_{\Lambda}(U^{\otimes k}) \cong (\Endo_{\Lambda}(U)^{\otimes k})^*$ (see (7) of \S \ref{iso}). Then from Theorem \ref{t:sergeev}, it follows that, under this isomorphism the $Q(U)$-invariant elements of $(\Endo_{\Lambda}(U)^{\otimes k})^*$ correspond to $\Lambda Ser_k \subseteq \Endo_{\Lambda}(U^{\otimes k})$.

2. Let $T \in \Endo_{\Lambda}(U^{\otimes k})$. Then, again it follows from Theorem \ref{t:sergeev} that $T$ commutes with $\rho(g)$ for all $g \in Q(U)$ if and only if $T$ is in the image of $\Lambda Ser_k$. 






3. Recall that we defined the supertarce map $str:\Endo_{\Lambda}(U) \rightarrow \Lambda$ by $$str(v \otimes \phi)=(-1)^{|v||\phi|}\phi(v).$$ For $A \in \Endo_{\Lambda}(U)_0$, we define the queer trace as $qtr(A)=str(PA)$.

4. Let $\mathcal I$ be the set of elements of $\Endo_{\Lambda}(U)_0$ which commute with $P$. Note that $Q(V)$ equals the set of invertible elements of $\mathcal I$. We have $str(A)=0$ for $A \in \mathcal I$ as $P^2=-I$ and $str(PAP)=str(A)$. 

5. Let $A_1, \cdots A_d \in \mathcal I$ and $\epsilon_1, \cdots \epsilon_d \in \mathbb Z_2$. Then it follows from (4) that if $\sum_i \epsilon_i$ is even then $str(P^{\epsilon_1}A_1 \cdots P^{\epsilon_d}A_d)=0$ and if $\sum_i \epsilon_i$ is odd then $str(P^{\epsilon_1}A_1 \cdots  P^{\epsilon_d}A_d)=qtr(A_1 \cdots A_k)$.

\section{Extending results of \S4 to $Q(U)$-invariants}

Given a mixed tensor space $W=\dsum_{i=1}^s(U_{b_i}^{t_i})$ as in Section~4, the restriction to $Q(U)$ of the $\gl(U)$-action on $W$ and $T^r(W^*)$ gives a natural $Q(U)$-action on $W$ and $T^r(W^*)$ respectively. As the $\gl(U)$-action commutes with the natural $\Sym_r$-action on $T^r(W^*)$, so also does the $Q(U)$-action. Hence, as seen in \S3, the subspace $e(r)T^r(W^*)$ is a $Q(U)$-invariant subspace which is isomorphic to $S^r(W^*)$. Thus $S^r(W^*)$ has a $Q(U)$-action defined on it and the map $\varpi_r:T^r(W^*)\to S^r(W^*) $ is $Q(U)$-equivariant with respect to this action.

With this notation, to obtain an analogue of Theorem~\ref{t:picture} we will need to associate further invariants in addition to the graded picture invariants defined in \S4. Denoting the $\Lambda$-basis of $U$ corresponding to $\{v_1,\ldots, v_n,v_{-{1}},\ldots, v_{-{n}}\}$ also by the same symbols, we note as in \S4.2 that $S(W^*)$ is a polynomial algebra with supercommuting variables $T(i)_{l_1\ldots l_{b_i}}^{u_1\ldots u_{t_i}}$ where $l_j$ and $u_j$ take values in $\{1,\ldots, n, -{1},\ldots, -{n}\}$. Let $(m_1,\ldots, m_s)$ be an an $s$-tuple such that $\sum_{k=1}^sm_kt_k=\sum_{k=1}^sm_kb_k=N$, then associated to the elements $c_{h}\in \Lambda Ser_N$ the picture invariant $\varphi_{c_h}$ is defined by 
$$\sum_{r_1,\ldots, r_N\in \{1,\ldots, n,-1, \cdots ,-n\}}(-1)^{\hat{q}(h,r_1,\ldots, r_N, m_1,\ldots, m_s)}\prod_{i=1}^{s}\prod_{j=1}^{m_i}T(i)_{r_{(\sum_{p<i}m_pb_p+(j-1)b_i+1)}\cdots r_{(\sum_{p<i}m_pb_p+jb_i)}}^{{r'_{(\sum_{p<i}m_pt_p+(j-1)t_i+1})}\cdots r'_{(\sum_{p<i}m_pt_p+jt_i)}}$$
where $r'_j=r_j$ for $j\neq h$ and $r'_h=-{r_h}$. The value of $\hat{q}(h,r_1,\ldots, r_N, m_1,\ldots, m_s)$ is given by  $\gamma(\mu^{-1}.(I,I^h),\mu^{-1})\gamma((I,I^h),(\tau.\nu)^{-1})\hat{\gamma}(I,h)$, where $I=(r_1,\ldots,r_N)$ and $\hat{\gamma}(I,h)=(-1)^{|v_{r_1}|+\cdots +|v_{r_{h-1}}|+|v_{r_h}|}$ and $I^h=(r_1,\ldots, -r_h,\ldots, r_N).$

The above formula is obtained by going through the isomorphisms in the proof of Theorem~\ref{t:picture}, as was done there for the case of $\sigma\in \Sym_N$. These isomorphisms are $\gl(U)$-equivariant and hence also $Q(U)$-equivariant. Replacing by $-{1},\ldots,-{n}$ the indices $n+1,\ldots, n+n$ wherever they appear in the isomorphisms listed in the proof of Theorem~\ref{t:picture} we get the generators of $\left(\tensor_{i=1}^sS^{m_i}((U_{b_i}^{t_i})^*)\right)^{Q(U)}$. Using the same argument as given there, it may be first noted that there are invariants in $\tensor_{i=1}^sS^{m_i}((U_{b_i}^{t_i})^*)$ if and only if $\sum_{k=1}^sm_kt_k=\sum_{k=1}^sm_kt_k=N$. For such an $s$-tuple $(m_1,\ldots, m_s)$,  we start with the generators of $(\End_\Lambda (U^{\tensor N}))^{Q(U)}$ prescribed by Theorem~\ref{t:sergeev}. Then as done in Theorem~\ref{t:picture}, going via the isomorphisms given there, the $\sigma\in \Sym_N$ give rise to the polynomials $\varphi_\sigma$ as defined in \S\ref{s:picture}. Now by the Schur-Weyl-Sergeev duality, the elements $c_h$, $h=1,\ldots, N$ give $Q(U)$-equivariant endomorphisms on $U^{\tensor N}$ which correspond under the natural isomorphism listed in (7) of \S \ref{iso}, $\End_{\Lambda}(U^{\tensor N})\iso \left( U^{\tensor N}\tensor_{\Lambda} (U^*)^{\tensor N}\right)^*$, to the linear map given by $\underline{v}\tensor \underline{f}^*\mapsto ev[\nu.\tau.(c_h.\underline{v}\tensor \underline{f}^*)]$. This may be expressed in term of the standard dual basis of $ U^{\tensor N}\tensor_{\Lambda} (U^*)^{\tensor N}$ by
\begin{equation}\label{eq:5.1}
    \sum_{L=(l_1,\ldots, l_N)}\hat\gamma(L,h)\gamma((L^h,L^h),\tau^{-1}) \gamma(\tau.(L^h,L^h),\nu)T_{L,L^h}.
\end{equation}
where $L^h=(l_1',\ldots, l_N')$ given by $l'_i=l_i$ unless $i=h$ and $l'_h=-{l_h}$. Now going through the isomorphisms in (\ref{eq:1}), (\ref{eq:2}), it can be seen that this element corresponds to $\varphi_{c_h}$ with the sign $(-1)^{\hat{q}(h,r_1,\ldots ,r_N, m_1,\ldots, m_s)}$ given by 
$\gamma(\mu^{-1}.(I,I^h),\mu^{-1})\gamma((I,I^h),(\tau.\nu)^{-1})\hat{\gamma}(I,h)$. We have thus obtained the following analogue of Theorem~\ref{t:picture}:
\begin{theorem}\label{t:queer}
    Given a mixed tensor space $W=\dsum_{i=1}^sU_{b_i}^{t_i}$, the graded picture invariants as in \emph{(\ref{eq:picture})} and \emph{(\ref{eq:5.1})} span $S(W^*)^{Q(U)}$.\hfill\qed
\end{theorem}

For an element $\alpha\in\Lambda Ser_N\subset \End_\Lambda(U^{\tensor N)}$, $\hat{\alpha}$ denotes the endomorphism $\alpha\tensor id$ on $U^{\tensor N}\tensor_\Lambda (U^*)^{\tensor N}$ The following lemma is the analogue of Lemma~\ref{l:trace} in the case of $Q(U)$:
\begin{lemma}\cite[Lemma~5.8]{Berele2}
    For a $\sigma\in \Sym_N$ and $c_1^{\varepsilon_1}\ldots c_N^{\varepsilon_N}\in \mathbf Cl_N$. The $Q(U)$-invariant map $ev((\nu\tau.\widehat{(\sigma\tensor c_1^{\varepsilon_1}\ldots c_N^{\varepsilon_N})}.\tau^{-1})(v_1\tensor \phi_1\tensor \cdots \tensor v_N\tensor \phi_N))$ 
    corresponds to the trace monomials $tr_\sigma(P^{\varepsilon_1}v_1\tensor \phi_1,\ldots, P^{\varepsilon_N}v_N\tensor \phi_N)$, upto a sign. Further, both the maps agree when restricted to the degree $0$ part, $((U_1^1)_0)^{\tensor N}$.
\end{lemma}
\begin{proof}
Since $(\sigma\tensor 1).(id\tensor c_1^{\varepsilon_1}\ldots c_N^{\varepsilon_N})=\sigma\tensor c_1^{\varepsilon_1}\ldots c_N^{\varepsilon_N}$  we consider
action of $id\tensor c_1^{\varepsilon_1}\ldots c_N^{\varepsilon_N}$ followed by 
 $\sigma\tensor 1$ from the left on $\tau^{-1}.(v_1\tensor \phi_1\tensor \cdots \tensor v_N\tensor \phi_N)$. This gives the vector $(-1)^{p(\underline v, \underline \phi)}(P^{\varepsilon_{\sigma^{-1}(1)}}v_{\sigma^{-1}(1)}\tensor \ldots \tensor P^{\varepsilon_{\sigma^{-1}(N)}}v_{\sigma^{-1}(N)}\tensor\phi_1\tensor\ldots\tensor \phi_N)$. Applying $\nu\tau$ followed by the evaluation we get $\prod_{i=1}^N\phi_i(P^{\varepsilon_{\sigma^{-1}(i)}}v_{\sigma^{-1}(i)})$. Then from Proposition~\ref{l:restitution2} we note that this is equal to $tr_\sigma(P^{\varepsilon_{1}}v_{1}\tensor \phi_1,\ldots, P^{\varepsilon_{N}}v_{N}\tensor \phi_N)$.
\end{proof}
\newcommand{\qtr}{\textup{qtr}}
\begin{corollary}
     The invariant polynomials for the simultaneous action of $Q(U)$ on $\oplus_{i=1}^s(U_1^1)_0$ is spanned by the trace monomials $\tr_\sigma$ given by \[\qtr(A_{f(i_1)}\cdots A_{f(i_r)})\qtr(A_{f(i_{r+1})}\cdots A_{f(i_t}))\cdots\] where $A_1,\ldots, A_s\in\oplus_{i=1}^sU_1^1$, $\sigma=(i_1 \ldots i_r)(i_{r+1} \ldots i_t)\cdots\in \Sym_n$ and a map $f:\{1,\ldots, n\}\to \{1,\ldots, s\}$.   
\end{corollary}
\begin{proof}
    The above lemma and Theorem~\ref{t:queer} put together give the generators to be $tr_\sigma(P^{\varepsilon_1}v_1\tensor \phi_1,\ldots, P^{\varepsilon_N}v_N\tensor \phi_N)$. Then (5) of \S \ref{q.fact}, gives the required result.
\end{proof}

\end{document}